\documentclass[11pt,reqno,a4paper]{amsart}

\usepackage{mathrsfs, graphics, xfrac}
\usepackage{amssymb}
\usepackage[notref,notcite,final]{showkeys}
\usepackage{hyperref}\hypersetup{colorlinks=true, citecolor=blue}
\usepackage{graphicx, epstopdf}\graphicspath{{./immaginiprova/}}
\usepackage{bbm}
\numberwithin{equation}{section}

\newtheoremstyle{mytheorem}
{}
{}
{\it}
{\parindent}
{\bf}
{.}
{ }
{\thmnumber{#2.~}\thmname{#1}\thmnote{~\rm#3}}

\newtheoremstyle{myremark}
{}
{}
{\rm}
{\parindent}
{\bf}
{.}
{ }
{\thmnumber{#2.~}\thmname{#1}\thmnote{~\rm#3}}

\newtheoremstyle{myparagraph}
{}
{}
{\rm}
{\parindent}
{\bf}
{.}
{ }
{\thmnumber{#2.~}\thmname{#1}\thmnote{#3}}

\theoremstyle{mytheorem}

\newtheorem{theorem}[subsection]{Theorem}
\newtheorem{definition}[subsubsection]{Definition}
\newtheorem{lemma}[subsection]{Lemma}

\theoremstyle{myremark}
\newtheorem{remark}[subsection]{Remark}

\theoremstyle{myparagraph}

\newtheorem*{parag*}{}

\makeatletter
\def\@secnumfont{\sc}
\def\section{\@startsection{section}{1}%
\z@{1.5\linespacing\@plus .2\linespacing}{.7\linespacing}%
{\normalfont\sc\centering}}
\makeatother

\makeatletter
\def\ps@headings{\ps@empty
 \def\@evenhead{%
  \setTrue{runhead}%
  \normalfont\footnotesize
  \rlap{\thepage}\hfil
  \def\thanks{\protect\thanks@warning}%
  \leftmark{}{}\hfil}%
 \def\@oddhead{%
  \setTrue{runhead}%
  \normalfont\footnotesize\hfil
  \def\thanks{\protect\thanks@warning}%
  \rightmark{}{}\hfil \llap{\thepage}}%
\let\@mkboth\markboth}
\makeatother
\makeatletter
\renewenvironment{proof}[1][\proofname]{\par
  \pushQED{\qed}%
  \normalfont \topsep6\p@\@plus6\p@\relax
  \trivlist
  \itemindent\normalparindent
  \item[\hskip\labelsep
    \bfseries
    #1\@addpunct{.}]\ignorespaces
}{%
  \popQED\endtrivlist\@endpefalse
}
\providecommand{\proofname}{Proof}
\makeatother
\newcommand{\cone}{\times\!\!\!\!\!\times\,} 

\newcommand{\Flat}{\mathbb{F}}
\newcommand{\Mass}{\mathbb{M}}
\newcommand{\TP}{\textbf{TP}}
\newcommand{\OTP}{\textbf{OTP}}
\newcommand{\Prob}{\mathscr{P}}
\newcommand{\wto}{\rightharpoonup}
\newcommand{\wstarto}{\stackrel{*}{\rightharpoonup}}
\newcommand{\R}{\mathbb{R}}

\newcommand{\Sk}{\mathcal{S}}

\newcommand{\MM}{\mathbb{M}^\alpha}
\newcommand{\N}{\mathbb{N}}

\newcommand{\Q}{\mathbb{Q}}
\newcommand{\D}{\mathscr{D}}

\newcommand{\Haus}{\mathscr{H}}

\newcommand{\M}{\mathscr{M}}
\newcommand{\Mpiu}{\mathscr{M}_+}

\newcommand{\3}{1}

\newcommand{\length}{{\rm length}}

\newcommand{\Lip}{\mathrm{Lip}}

\newcommand{\dist}{\mathrm{dist}}
\newcommand{\supp}{\mathrm{supp}}

\newcommand{\wrt}{w.r.t.\ }

\newcommand{\e}{\varepsilon}
\newcommand{\eee}{\varepsilon}
\newcommand{\dV}{d_V\kern-1pt}

\newcommand{\trait}[3]{\vrule width #1ex height #2ex depth #3ex}

\newcommand{\trace}{\mathchoice%
  {\mathbin{\trait{.12}{1.2}{.03}\trait{.8}{0.09}{0.03}}}
  {\mathbin{\trait{.12}{1.2}{.03}\trait{.8}{0.09}{0.03}}}
  {\mathbin{\hskip.15ex\trait{.09}{.84}{0.02}\trait{.56}{.07}{.02}}\hskip.15ex}
  {\mathbin{\trait{.07}{.6}{.01}\trait{.4}{.06}{.01}}}}

\makeatletter
\@addtoreset{footnote}{section}
\makeatother

\newenvironment{itemizeb}
{\begin{itemize}\itemsep=2pt}{\end{itemize}}

\begin{document}

	%
\pagestyle{empty}
\pagestyle{myheadings}
\markboth%
{\underline{\centerline{\hfill\footnotesize%
\textsc{Maria Colombo, Antonio De Rosa, and Andrea Marchese}%
\vphantom{,}\hfill}}}%
{\underline{\centerline{\hfill\footnotesize%
\textsc{On the well-posedness of branched transportation}%
\vphantom{,}\hfill}}}

	%
\thispagestyle{empty}

~\vskip -1.1 cm

	%

\vspace{1.7 cm}

	%
{\large\bf\centering
On the well-posedness of branched transportation\\
}

\vspace{.6 cm}

	%
\centerline{\sc Maria Colombo, Antonio De Rosa, and Andrea Marchese}

\vspace{.8 cm}

{\rightskip 1 cm
\leftskip 1 cm
\parindent 0 pt
\footnotesize

	%
{\sc Abstract.}
We show in full generality the stability of optimal traffic paths in branched transport: namely we prove that any limit of optimal traffic paths is optimal as well.
This solves an open problem in the field (cf. Open problem 1 in the book \emph{Optimal transportation networks}, by Bernot, Caselles and Morel), which has been addressed up to now only under restrictive assumptions.


\par
}


\section{Introduction}

This paper deals with optimizers of the branched transportation problem. Given a source $\mu^-$ and a target $\mu^+$, positive measures on $\R^d$ with compact support, a \emph{traffic path} transporting $\mu^-$ onto $\mu^+$ is given by a $1$-rectifiable current $T$ whose boundary $\partial T$ is $\mu^+-\mu^-$. This can be identified with a vector-valued measure $T=\vec T(\theta\Haus^1\trace E)$ (with unit vector field $\vec T$ and non-negative multiplicity $\theta$), supported on a bounded set $E\subset \R^d$, which is contained in a countable union of curves of class $C^1$ and having distributional divergence ${\mbox{div}}~T=\mu^--\mu^+.$
Given a parameter $\alpha \in (0,1)$, quantifying the convenience of grouping particles during the transportation, we consider the \emph{$\alpha$-mass} of $T$
\begin{equation}
\label{eqn:alphamass}
\Mass^\alpha(T):=\int_E \theta(x)^{\alpha}d\Haus^1(x),
\end{equation}
and the minimal transport energy to connect $\mu^-$ to $\mu^+$
\begin{equation}\label{mainp}
W^\alpha(\mu^-,\mu^+):=\inf\{\Mass^\alpha(T): \mbox{$T$ is a traffic path transporting $\mu^-$ onto $\mu^+$}\}.
\end{equation}
The optimizers in the minimization problem are called optimal traffic paths; the set of optimizers is denoted by $\OTP(\mu^-,\mu^+)$.
The existence of solutions is obtained by direct methods and in general one does not expect uniqueness. Arguably the main open question concerning the well-posedness of the problem, of special relevance in view of numerical simulations, 
 is whether or not the optima are \emph{stable} with respect to variations of the initial and final distribution of mass. In other words, we ask if the limit of suitable sequences of optima (with respect to the usual notion of convergence of vector-valued measures denoted by $T_n\overset{*}{\rightharpoonup} T$) is still an optimum.
%
%

The main result of our paper provides a positive answer to this question, raised in \cite[Problem 15.1]{BCM}, for every $\alpha \in (0,1)$. 
\begin{theorem}[(Stability of optimal traffic paths)]\label{thm:main}
		Let $\alpha \in(0,1)$,  $\mu^-,\mu^+
		$ be mutually singular positive measures on $\overline{B(0,R)}$, $R>0$, satisfying $\mu^-(\R^d)=\mu^+(\R^d)$. 
		Let $\{\mu^-_n\}_{n\in \N}, \{\mu^+_n\}_{n\in \N}$ be positive measures on $\overline{B(0,R)}$ such that $\mu^-_n(\R^d)=\mu^+_n(\R^d)$ for every $n \in \N$ and 
		\begin{equation}\label{hp:supp-n-convergence}
		\mu^\pm _n \overset{*}{\rightharpoonup} \mu^\pm,		\end{equation}
		and assume there exist $T_n\in \OTP(\mu^-_n,\mu^+_n)$ optimal traffic paths satisfying 
		\begin{equation}
		\label{hp:energy-bound}
		\sup_{n\in \N}\MM(T_n)<\infty.
		\end{equation}
		Then, the (non-empty) family of subsequential weak-$*$ limits of $T_n$ is contained in $\OTP(\mu^-,\mu^+)$. 
	\end{theorem}
	
	\begin{remark}[($H$-masses)]
With minor changes, 
 Theorem \ref{thm:main} holds true for every \emph{$H$-mass}. Namely we can replace the integrand $x \mapsto x^\alpha$ in \eqref{eqn:alphamass} with a general function $H:\R\to[0,\infty)$ which is even, sub-additive, lower semi-continuous, monotone non-decreasing in $(0,+\infty)$, continuous in $0$ and satisfies $H(0)=0$. These functionals have been widely studied (see e.g. \cite{White1999,depauwhardt,flat-relax,BW,CFM,MW}).
	The interest is twofold: firstly a general formulation of the branched transportation problem allows to consider several interesting models, which are relevant for applied mathematics and numerical approximations as in \cite{BW}, secondly the possibility to prove the result in such generality shows the flexibility and the robustness of our strategy, which does not employ any peculiar property of the function $x \mapsto x^\alpha$. In Remark \ref{hmas} we detail how to modify the proof of Theorem \ref{thm:main} to include such generalization.
	\end{remark}
	\subsection{Background}
In the case of discrete measures $\mu^-$ and $\mu^+$, the minimization problem \eqref{mainp} was suggested by Gilbert \cite{Gilbert}, who proposed finite directed weighted graphs $G$ as transportation networks. For arbitrary measures $\mu^-$ and $\mu^+$, two generalizations of the the Gilbert problem have been proposed. On one hand, the above description in terms of traffic paths is due to Xia \cite{Xia,xia2}, and it is related to a problem which arises in the characterization of weakly approximable Sobolev maps with values in a manifold \cite{hrActa}.
On the other hand, a different model was introduced and studied in \cite{MSM,BCM}: here the transportation networks (called \emph{traffic plans}) consist in measures
on the set of Lipschitz paths, where each path represents the trajectory of a single particle.

In both models, the existence of optimizers in the minimization problem has been established \cite{Xia,MSM,BeCaMo,brabutsan,Pegon} (see also the reference book \cite{BCM}). The correspondence between traffic plans and traffic paths can be established by means of Smirnov's theorem on the structure of acyclic, normal 1-dimensional currents \cite{Smirnov93}. Indeed, the two
formulations were proved to be equivalent (see \cite{BCM,Pegon} and references therein).
Under some restrictions on $\alpha, \mu^-$ and $\mu^+$, optimizers exhibit regularity properties both in the interior (roughly speaking, they are locally finite graphs) and close to their boundary, that is the supports of $\mu^\pm$ \cite{xia2,MR2250166,DevSolElementary,morsant,xiaBoundary,BraSol}.

The models described above can be used and generalized to describe a variety of problems related to branched transportation: for instance, one can study the mailing problem \cite{BCM} (for which the first stability result was proved in \cite{CDRM3}), the urban planning model \cite{BranK}, including two different regimes of transportation, or the recent multi-material transport problem \cite{MMT,MMST}, allowing simultaneous transportation of different goods or commodities. Recently, shape optimization problems related to the functional \eqref{eqn:alphamass} were analysed in \cite{PeSaXia,BrSun} and similar branching structures are observed in superconductivity models and for minimizers of Ginzburg-Landau type functionals, see for instance \cite{chok3,chok,chok1,chok2,con}. 
 
Explicit optima are known only in few (mainly discrete) cases; for this reason, some effort has been put in developing numerical strategies to compute minimizers, for instance in term of phase-field approximations \cite{OuSan,BCF,BLS}, in the spirit of numerical calibrations \cite{massoubo,BOO}, or exploiting the convex nature of different formulations of some aspects of the problem (which is overall highly nonconvex) \cite{marmass,marmass1,BranRS}. 

\begin{remark}[(Stability in previous works)]\label{remarkone}
The answer to the stability question was previously known for $\alpha  \in (1-\sfrac1d,1]$. In this case, a simple argument relies on the fact that the minimal transport energy $W^\alpha(\nu_n,\nu)$ metrizes the weak-$*$-convergence of probability measures $\nu_n\overset{*}{\rightharpoonup} \nu$  (see \cite[Lemma 6.11 and Proposition 6.12]{BCM}). This property is false for $\alpha\leq 1-\sfrac1 d$, as shown in \cite{CDRM1}. The threshold
$\alpha=1-\sfrac{1}{d}$
appears also because for $\alpha$ above this value any two probability measures with compact support in $\R^d$ can be connected with finite cost. The same threshold is then recurrent in other results: for instance, above the threshold interior regularity holds (see \cite[Theorem 8.14]{BCM}) and a possible proof is obtained using the stability property. 
\end{remark}

\subsection{Strategy of the proof} In analogy with previous works \cite{BeCaMo, BCM,CDRM1}, to prove Theorem~\ref{thm:main} we assume by contradiction that $T$ is not optimal, denote $T_{opt}$ a minimizer, and we construct a better competitor for $T_n$ ($n$ large enough) by ``sewing'' a small portion of the traffic path $T_n$ with a large portion of $T_{opt}$.
In the following we shortly describe some of the main ideas and difficulties behind the proof of Theorem~\ref{thm:main}.

\subsubsection{Lagrangian description of traffic paths}
By means of Smirnov theorem we decompose the optimal path $T_n$ as a superposition of curves without cancellations. At difference from previous works, 
our energy competitor for $T_n$ is not only expressed in Lagrangian terms as a cut and paste of trajectories, to exploit the full power of the \emph{slicing} operation defined for currents (see \S \ref{s:slicing}).

\subsubsection{Cancellations in the Lagrangian description of $T$}
A technical difficulty for our construction is related to the fact that, although the limit of the Lagrangian descriptions of $T_n$ provides a Lagrangian description of $T$, the latter could contain cycles and cancellations at the level of currents. This issue did not appear in \cite[Theorem 1.2]{CDRM1} because there the convergence $T_n \wstarto T$ was not necessary to obtain a cheap connection of the slices.
To overcome this and obtain a lower semi-continuity result which keeps track in the limit of those Lagrangian trajectories which have opposite orientations and therefore they would cancel at the Eulerian level, we employ some ideas from the theory of currents with coefficients in normed groups (see \S \ref{s:g-currents}).

\subsubsection{Sewing trajectories}
Lemma \ref{high_multiplicity} shows that, even though $\MM$ does not metrize the weak-$*$ convergence of measures for $\alpha$ below the critical threshold (as explained in Remark~\ref{remarkone}), this holds true on the class of atomic measures with uniformly bounded energy (the energy of an atomic measure is defined in \eqref{defn:alpha-mass-meas}). 
This lemma is applied to the slices of some portions of $T_n$ and $T$ along the boundary of small cubes and it allows us to have a cheap connection between $T_n$ and $T$ in proximity of the boundary. For such operation we need to exploit the convergence of the slices of $T_n$ to the slices of $T$: for this reason we cannot directly connect the trajectories of $T_n$ to the trajectories of $T_{opt}$. 

\subsubsection{Comparison with previous strategies}
In \cite[Theorem 1.2]{CDRM1} we employed a dimension-reduction argument to cut the trajectories of $T_n$ and glue them with the trajectories of $T_{opt}$.
There are three substantial differences in the approach we adopt in the present paper: firstly, in the previous work we guaranteed the smallness of the connection by making it act on a $d-1$ dimensional surface (hence the bound $\alpha>1-\sfrac{1}{d-1}$); 
secondly, to guarantee the smallness of the connection we required that $\mu^\pm$ were supported on an $\Haus^1$-null set;
lastly, while in \cite[Theorem 1.2]{CDRM1} the connection acted on Lagrangian trajectories, in this paper we need to perform the slicing at the Eulerian level of currents, possibly introducing cancellations in mass.

\section{Notation and preliminaries}\label{s:notation}
\subsection{Sets and Measures} 
We add below a list of frequently used notations:
\begin{itemizeb}\leftskip 0.8 cm\labelsep=.3 cm
\item[${\bf e_1},\dots,{\bf e_d}$]
standard basis of $\R^d$;

\item[$B(x,r)$]
\emph{open} ball with center $x$ and radius $r$;

\item[$\overline A$]
\emph closure of the set $A$;

\item[$1_E$] 
characteristic function of a set $E$, taking values $0$ and $1$;

\item[${\rm Im}\gamma$] 
image (or support) of a curve $\gamma$;

\item[$|v|$]
Euclidean norm of a vector $v\in \R^d$;

\item[${\rm dist}(x,A)$] 
$:=\inf_{y\in A}\{|x-y|\}$, distance between the point $x$ and the set $A$; we also denote ${\rm dist}(A,B):=\inf_{y\in A}\{{\rm dist}(y,B)\}$ and $B(A,\rho):=\{x:\dist(x,A)<\rho\}$;

\item[$\M_+(Y)$]
set of positive Radon measures on the space $Y$; we use $\mathscr P(Y)$ for the subset of probability measures; 

%
%
%
\item[$f\mu $]
measure associated to a measure $\mu$ and a 
function $f$, namely 
$[f\mu](E):=\int_E f \, d\mu$;
\item[$\mu\trace E$]
$:=1_E\mu$, restriction of a measure $\mu$ to a set $E$;
\item[$f_\#\,\mu$]
push-forward of a measure $\mu$ on $Y$
according to a map $f:Y\to Y'$, that is,
the measure on $Y'$ given by
$[f_\#\,\mu](E):=\mu(f^{-1}(E))$;

\item[$|\mu|$] 
total variation measure associated to a real- or vector-valued 
measure $\mu$; we call \emph{positive} and \emph{negative part} of a real-valued measure $\mu$ respectively the measures $\sfrac{1}{2}(|\mu|+\mu)$ and $\sfrac{1}{2}(|\mu|-\mu)$;
\item[$\supp (\mu)$]
support of $\mu$; we say that $\mu$ is \emph{supported} on $E$ if $|\mu|(Y\setminus E)=0$; we say that two measures $\mu$ and $\nu$ are \emph{mutually singular} if $\mu$ is supported on a set $E$ such that $|\nu|(E)=0$;

\item[$\Mass(\mu)$]
$:=|\mu|(Y)$, mass of a measure $\mu$ on a space $Y$;
\item[$\mu\leq\nu$]
means that $\mu(A)\leq\nu(A)$ for every Borel set $A$;

\item[$\delta_x$]
Dirac delta at the point $x$;

\item[$\Haus^k$]
$k$-dimensional Hausdorff measure;

\item[$L^p(\mu)$]
space of $p$-integrable functions \wrt $\mu$; we also use $L^p(\mu; V)$ for $p$-integrable functions with values in the normed space $V$.
\item[$\|\cdot \|_p$]
$L^p$-norm; we use $\|\cdot\|_\infty$ also to denote the supremum norm;

\item[$\mu_n\wstarto\mu$]
denotes the weak-$*$ convergence of measures, that is $\int f d\mu_n\to\int f d\mu$ for every $f\in C^0_c$.

\end{itemizeb}




\subsection{Rectifiable sets and currents}\label{ss:currents} We recall here the basic terminology related to $k$-dimensional rectifiable sets and currents. We refer the reader to the introductory presentation given in the standard textbooks \cite{SimonLN}, \cite{KrantzParks} and to the most complete treatise \cite{FedererBOOK}. For the purposes of this paper, we point out that in \cite{CDRM1} the same was used and more extensively presented in the context of branched transport.

For $k=0,1,\dots,d$, a set \(E\subset \R^d\) is said 
\emph{\(k\)-rectifiable} if it can be covered, up to an \(\Haus^k\)-negligible set, by countably many $k$-dimensional submanifolds of class \(C^1\). 

In the sequel we use the following notation: 
\begin{itemizeb}\leftskip 0.8 cm\labelsep=.3 cm
\item[${\rm{Tan}}(E,x)$]
tangent $k$-plane to the $k$-rectifiable set $E$ at the point $x$ (defined at $\Haus^k$-a.e. $x\in E$);

\item[$\D^k(\R^d)$]
space of smooth and compactly supported differential $k$-forms on $\R^d$. The topology on $\D^k(\R^d)$ is analogous to the topology defined on the space of test functions with respect to which distributions are dual;

\item[$\D_k(\R^d)$]
\emph space of $k$-dimensional currents in $\R^d$, namely continuous linear functionals on $\D^k(\R^d)$;

\item[$\langle T,\omega\rangle$] 
duality pairing between a $k$-current $T$ and a $k$-form $\omega$. We use the same symbol for the duality pairing between a $k$-covector and a $k$-vector;

\item[$T_n \wto T$] 
weak-$*$ convergence of currents, namely $\langle T_n,\omega\rangle\to \langle T,\omega\rangle$ for every $\omega\in\D^k(\R^d)$;

\item[$\partial T$]
boundary of $T$, that is the $(k-1)$-dimensional current defined  via 
$\langle\partial T, \phi \rangle := \langle T, d\phi\rangle$ for every $\phi\in \D^{k-1}(\R^d)$;

\item[$\|\omega\|$]
$:=\sup_{x,\tau}\{\langle \omega(x),\tau\rangle$: $x\in\R^d$, $\tau$ is a unit simple $k$-vector$\}$ is the comass norm of the form $\omega$;

\item[$\Mass(T)$] 
$:=\sup_{\omega}\{\langle T, \omega\rangle$: $\|\omega\|\le 1\}$ is the mass of the current $T$;

\item[$T=\vec{T} |T|$]
representation of a current with finite mass (or a vector valued measure)\footnote{Even though currents with finite mass and vector valued measures can be naturally identified, the convergence of currents does not imply in general convergence of vector valued measures. This is the reason for using the two different symbols $\mu_n\wstarto\mu$ and $T_n\wto T$.}, namely 
$\langle T, \omega\rangle = \int_{\R^d} \langle\omega(x), \vec{T}(x)\rangle d\|T\|(x)
$, where $|T|\in\Mpiu(\R^d)$ and $\vec{T}$ is a unit $k$-vector field. In particular $\Mass(T)=\Mass(|T|)$;

\item[$\supp(T)$]
support of $T$ (in the distributional sense);

\item[$\mathbf{N}_k(\R^d)$]
normal currents, that is currents $T$ such that both $T$ and $\partial T$ have finite mass;

\item[$\partial_+T, \partial_-T$]
(for $T\in \mathbf{N}_1(\R^d)$) positive and negative part of the (finite) measure $\partial T$;

\item[$T \trace A$] 
restriction of a current $T$ with finite mass to the Borel set $A$, namely $\langle T \trace A, \omega\rangle := \int_{A} \langle\omega(x), \vec{T}(x)\rangle d|T|(x)$;

\item[$\Flat(T)$]
flat norm of the current $T$, that is $\Flat(T):=\inf\{\Mass(R)+\Mass(S): T=R+\partial S,\, R \in \D_k(\R^d),\,  S \in \D_{k+1}(\R^d) \}$;

\item[$\mathbf{R}_{k}(\R^d)$]
space of $k$-rectifiable currents, represented as $T=[E,\tau,\theta]$, which means
$\langle[E,\tau,\theta],\omega\rangle:= \int_{E} \langle\omega(x), \tau(x)\rangle \, \theta(x) d \Haus^k(x),$ where $E$ is a $k$-rectifiable set, $\tau(x)$ is a unit, simple $k$-vector field spanning ${\rm{Tan}}(E,x)$ for $\Haus^k$-a.e $x\in E$, and $\theta\in L^1_{loc}(\Haus^k\trace E)$; in particular $\Mass(T)=\int_{E} |\theta(x)| d \Haus^k(x)$;

\item[$\Mass ^\alpha (T)$]
$:=\int_E |\theta|^\alpha(x) d\Haus^k(x)$ is the $\alpha$-mass of $T$, where $\alpha \in (0, 1]$ and $T=[E,\tau,\theta]$. We set $\Mass ^\alpha (T)=+\infty$ for every $T\in \mathbf{N}_{k}(\R^d)\setminus \mathbf{R}_{k}(\R^d)$.
\end{itemizeb}
\begin{remark}[(Flat norm and weak-$*$ convergence)]\label{rmk_andrea}
In general $\Flat(T_n-T)\to 0$ implies that $T_n\wto T$. If the $T_n$'s are all supported on a common compact set, and they have equi-bounded masses and masses of the boundaries the reverse is also true. This fact can be easily deduced from \cite[Theorem 4.2.17(1)]{FedererBOOK}.
\end{remark}

\subsection{Traffic paths}

Fix $R>0$. From now on, by $X$ we denote the closed ball of radius $R$ in $\R^d$ centered at the origin. Following \cite{Xia} and 
\cite{BCM}, given two positive measures $\mu^-,\mu^+ \in \M_+(X)$ with the same total variation, we define the set $\TP(\mu^-,\mu^+)$ of the \emph{traffic paths} connecting $\mu^-$ to $\mu^+$ as $$\TP(\mu^-,\mu^+):=\{T\in\mathbf{N}_1(\R^d): \supp(T)\subset X, \partial T=\mu^+-\mu^-\},$$
and the \emph{minimal transport energy} associated to $\mu^-,\mu^+$ as
$$W^{\alpha}(\mu^-,\mu^+):= \inf \{\MM(T): T \in \TP (\mu^- ,\mu^+)\}.$$
Moreover we define the set of \emph{optimal traffic paths} connecting $\mu^-$ to $\mu^+$ by 
\begin{equation}
	\label{eqn:otp}
	\OTP (\mu^- ,\mu^+):=\{T \in \TP (\mu^- ,\mu^+) : \MM(T)=W^\alpha(\mu^-,\mu^+) \}.
\end{equation}
%
As observed in \cite[Proposition 2.5]{CDRM1}, in order to minimize the $\alpha$-mass among currents with boundary in $X$, it is not restrictive to consider only currents supported in $X$. 
\subsection{Structure of optimal traffic paths and good decompositions}\
In the class of rectifiable 1-currents, some basic objects are given by the ones associated to Lipschitz simple curves with finite length. The aim of this subsection is to describe the so called ``superposition principle'' according to which every acyclic normal 1-current can be written as a weighted average of such curves.

We denote by $\Lip$ the space of $1$-Lipschitz curves $\gamma: [0,\infty) \to \R^d$ which are eventually constant (and hence of finite length). For $\gamma\in\Lip$ we denote by $T_0(\gamma)$ and $T_\infty(\gamma)$ the values
 $$T_0(\gamma):=\sup\{t:\gamma \mbox{ is constant on }[0,t]\}
 \qquad T_\infty(\gamma):=\inf\{t:\gamma \mbox{ is constant on }[t,\infty)\}.$$ 
 Given $\gamma \in \Lip$,
 we call $\gamma(\infty):=\lim_{t\to\infty}\gamma(t)$.
 We say that a curve $\gamma\in\Lip$ is \emph{simple} if $\gamma(s)\neq\gamma(t)$ for every $T_0(\gamma)\leq s<t\leq T_\infty(\gamma)$ such that $\gamma$ is non-constant in the interval $[s,t]$.
 
We associate canonically to each simple curve $\gamma\in \Lip$,
the rectifiable $1$-current 
$ I_\gamma:=[{\rm{Im}}\gamma,\sfrac{\gamma'}{|\gamma'|},1].
$
 It is easy to check that
 $\Mass(I_\gamma)=\Haus^1(\rm{Im}\gamma)$
 and
$\partial I_\gamma=\delta_{\gamma(\infty)}-\delta_{\gamma(0)}$;
since $\gamma$ is simple, if it is also non-constant, then $\gamma(\infty) \neq \gamma(0)$ and $\Mass(\partial I_\gamma)=2$.

A normal current $T\in \mathbf{N}_1(\mathbb{R}^d)$ is said \emph{acyclic} if there exists no non-trivial current $S$ such that
$$\partial S=0 \qquad \mbox{and} \qquad \Mass(T)=\Mass(T-S)+\Mass(S).$$

We recall a fundamental result of Smirnov (\cite{Smirnov93}) which establishes that every acyclic normal 1-current can be written as a weighted average of simple Lipschitz curves in the following sense.
 
\begin{definition}[\bf{(Good decomposition)}]\label{defn:GD}
Let $T\in \mathbf{N}_1(\mathbb{R}^d)$ be represented as a vector-valued measure $\vec T |T|$, and let $P \in \M_+(\Lip)$ be a finite positive measure, supported on the set of curves with finite length, such that 
\begin{equation}
\label{eqn:buona-dec}
T=\int_{\Lip} I_\gamma d P (\gamma),
\end{equation}
namely for every smooth compactly supported 1-form $\varphi: \R^d \to \R^d$ it holds
\begin{equation}
\label{eqn:good-dec-operativa}
\int_{\R^d} \langle\varphi, \vec T\rangle \,d| T|=
\int_{\Lip} \int_{0}^\infty \langle\varphi(\gamma(t)), \gamma'(t)\rangle\, dt \, d P(\gamma).
\end{equation}
We say that $P$ is a good decomposition of $T$ if $P$ is supported on non-constant, simple curves and satisfies the equalities
\begin{equation}
\label{eqn:buona-dec-mass-T}
\Mass(T) 
= \int_{\Lip} \Mass(I_\gamma) d P(\gamma)
= \int_{\Lip} \Haus^1({\rm{Im}}\gamma) d P(\gamma)
  \, ; 
\end{equation}
\begin{equation}
\label{eqn:buona-dec-mass-boundaryT}
\Mass(\partial T) 
= \int_{\Lip} \Mass(\partial I_\gamma) d P(\gamma)
= 2 P({\Lip})
  \, .
\end{equation}
\end{definition}

It has been shown in \cite[Theorem 10.1]{PaoliniStepanov} that optimal traffic paths $T\in \OTP(\mu^-, \mu^+)$ are acyclic, hence they admit such a good decomposition. In the next result, we collect some useful properties of good decompositions, whose proof can be found in {\cite[Proposition 3.6]{CDRM1}}.

\begin{theorem}[(Existence and properties of good decompositions){\cite[Theorem 5.1]{PaoliniStepanov1}} and {\cite[Proposition 3.6]{CDRM1}}]\label{t:propr_good_dec}
Let $\mu^-, \mu^+ \in \M_+(\R^d)$ and $T \in \OTP(\mu^-, \mu^+)$ with finite $\alpha$-mass. Then $T$ is acyclic and there is a Borel finite measure $P$ on $\Lip$ such that 
$P$ is a good decomposition of $T$. Moreover, if $P$ is a good decomposition of $T \in \mathbf{N}_1(\mathbb{R}^d)$  as in \eqref{eqn:buona-dec}, the following statements hold:

\begin{enumerate}
\item The positive and the negative parts of the signed measure $\partial T$ are
$\partial_- T = \int_{\Lip}\delta_{\gamma(0)} d P (\gamma)
$
and 
$\partial_+ T = \int_{\Lip}\delta_{\gamma(\infty)} d P (\gamma).
$
\item If $T= T[E, \tau, \theta]$ is rectifiable, then
$|\theta(x)| = P(\{\gamma: x \in {\rm{Im}}\gamma \})$ for $\Haus^1$-a.e. $x\in E$.
	\item For every $P' \leq P$ the representation
$	T' := \int_{\Lip} I_\gamma  dP'( \gamma )$
	is a good decomposition of $T'$; moreover, if $T= T[E, \tau, \theta]$ is rectifiable, then $T'$
	can be written as $T'=T[E, \theta',\tau]$ with $|\theta'| \leq \min\{|\theta|, P'(\Lip)\}$ and $\theta\cdot\theta'\geq 0$, $\Haus^1$-a.e..

%
\end{enumerate}
\end{theorem}

\begin{remark}[(Lagrangian description of the limit)]\label{rmk:limit}
Let $T_n \rightharpoonup T $ be a sequence of currents converging weakly-$*$ with uniformly bounded mass  and mass of the boundaries and let $P_n$ be good decompositions of $T_n$. Up to a subsequence, $P_n \wstarto P \in\Prob(\Lip) $ (thanks to \eqref{eqn:buona-dec-mass-boundaryT} and to \eqref{eqn:buona-dec-mass-T}, which ensure pre-compactness of the sequence of measures). Then $P$ might fail to be a good decomposition of $T$, but \eqref{eqn:buona-dec} remains valid.
Indeed, every smooth compactly supported $1$-form $\omega$, induces a continuous function on curves $\Lip \ni \gamma \to \langle I_\gamma, \omega\rangle$ and we can test both weak-$*$ convergences $T_n \rightharpoonup T $ and $P_n \wstarto P$ to obtain the equality.
\end{remark}

\section{Preliminary results}\label{sec:con}
Given a cube $Q\subset\R^d$ whose faces are parallel to the coordinate hyperplanes and $k\in\N$ we denote 
$$\Lambda(Q,k):=\{Q^\ell\}_{\ell=1}^{2^{kd}}$$ the collection of the $2^{kd}$ cubes obtained dividing each edge of $Q$ into $2^k$ subintervals of equal length. We denote by $$\Sk(Q,k):=\bigcup_{\ell=1}^{2^{kd}}\partial Q^\ell$$ the $(d-1)$-skeleton of the grid $\Lambda(Q,k)$.
Moreover we denote by $\rho Q^\ell$ the concentric cube to $Q^\ell$, with homothety ratio $\rho$.

Given two cubes $Q, R$ , we define $\Lip (Q, R)$ as the set of curves in $\Lip$ which start in $Q $ and end in $R$, namely
$$\Lip (Q, R):= \{ \gamma \in \Lip: \gamma(0) \in Q, \ \gamma(\infty)\in R\}.$$

Given an atomic measure $\mu\in\M_+(X)$ of the form $\mu=\sum_{i\in\N}\theta_i\delta_{x_i}$, we define its $\alpha$-mass
\begin{equation}
\label{defn:alpha-mass-meas}
\MM(\mu) =  \sum_{i\in\N}\theta_i^\alpha.
\end{equation}
The alpha mass of a real-valued atomic measure is simply the sum of the $\alpha$-mass of its positive and its negative part (the $\alpha$-mass of a measure is considered to be infinite if the measure is not atomic).
If $\mu$ is atomic and supported on a cube $Q_l(x)\subset \R^d$, centred at $x$ and with diameter $l$, the \emph{cone} over $\mu$ with vertex $x$, is defined as the 1-current
\begin{equation}\label{defn:cone}
x\cone\mu:=\sum_{i\in\N}\theta_iS_i,
\end{equation}
where $S_i$ is the 1-dimensional current canonically associated to the oriented segment connecting $x$ to $x_i$. It is easy to check that  
\begin{equation}\label{stimecono}
\partial(x\cone\mu)=\mu-\bigg(\sum_{i\in\N}\theta_i\bigg)\delta_x \quad\mbox{ and }\quad \MM(x\cone\mu)\leq l\cdot\MM(\mu).
\end{equation}

\begin{lemma}[(Existence of a sequence of negligible nested grids)]\label{lemmagriglia}
Let $Q\subset\R^d$ be a cube. Let $\{\mu_n\}_{n\in\N}\subset \M_+(Q)$ be a countable family of measures. Then there exists a cube $Q'\supset Q$ such that 
\begin{equation}\label{griglia}
\mu_n(\mathcal{S}(Q',k))=0, \quad \mbox{ for all } (k,n)\in\N^2.
\end{equation}
\end{lemma}
\begin{proof}
 Denote $\mu:=\sum_{n\in\N}2^{-n}\sfrac{\mu_n}{\Mass(\mu_n)}$. Let $Q''$ be cube such that $d(Q,(\R^d\setminus Q''))\geq 1$ and such that the edge length of $Q''$ is an integer number. For every $j=1,\dots,d$ and $k\in\N$ we denote 
$H_{j,k}$ the union of $2^k+1$ hyperplanes, orthogonal to ${\bf{e}}_j$, partitioning $Q''$ into $2^k$ slabs of equal volume. Denote also 
$$L_j:=\bigcup_{k\in\N}H_{j,k}.$$
Since $L_j+r{\bf{e}}_j$ is disjoint from $L_j+s{\bf{e}}_j$ whenever $r-s\in \R\setminus\Q$, then for every $j$ there exists $\rho_j\in[0,1]$ such that 
$$\mu(L_j+\rho_j{\bf{e}}_j)=0.$$
We conclude that $Q':=Q''+\sum_j\rho_j{\bf{e}}_j$ yields \eqref{griglia}.
\end{proof}

\subsection{A metrization property for $\MM$}\label{sec:metrizes}
We show that if a sequence $\mu_n$ of measures, satisfying a uniform bound on the $\alpha$-masses, weak-$*$ converges to a measure $\mu$, then the connection cost $W^\alpha(\mu_n,\mu)$ converges to zero, for every $\alpha \in (0,1)$ (compare with Remark~\ref{remarkone}, which requires instead $\alpha>1-\sfrac 1d$).
\begin{lemma}[(Metrization property for $\MM$)]\label{high_multiplicity}
Let $Q\subset\R^d$ be a cube and $C>0$. Let $\mu_n,\nu_n\in\Mpiu(Q)$ be atomic measures such that\footnote{We remind the reader that the symbol $\wto$ denotes the weak-$*$ convergence of $0$-currents. Under the assumptions of the lemma, this is equivalent to the weak-$*$ convergence of the associated real-valued measures.} $\mu_n- \nu_n\wto 0$ and for all $n\in\N$
$$\Mass(\mu_n)=\Mass(\nu_n), \quad \MM(\mu_n)+\MM(\nu_n)\leq C.$$ 	
Then $\lim_{n \to \infty} W^\alpha(\mu_n,\nu_n)= 0$.
\end{lemma}
\begin{proof}
By Lemma \ref{lemmagriglia} we can assume that, up to enlarging the cube Q, 
\begin{equation}\label{griglia2}
\mu_n(\Sk(Q,k))=\nu_n(\Sk(Q,k))=0, \quad \mbox{ for all } (k,n)\in\N^2.
\end{equation}
Now fix $k\in\N$ and $\gamma>0$; let $\{Q^\ell\}_{\ell=1,\dots,2^{kd}}$ be the cubes in $\Lambda(Q,k)$.
Denote by $\sigma_n$ the real-valued measure
$$\sigma_n:=\sum_{\ell=1}^{2^{kd}}\theta_\ell\delta_{x_\ell}\quad \mbox{ where $x_\ell$ is the barycenter of $Q_\ell$ and } \theta_\ell:=\nu_n(Q^\ell)-\mu_n(Q^\ell).$$
By \eqref{griglia2}, the assumption $\mu_n - \nu_n\wto 0$ yields
\begin{equation}\label{massaalfacono}
\Mass^\alpha(\sigma_n) =\sum_{\ell=1}^{2^{kd}}|\nu_n(Q^\ell)-\mu_n(Q^\ell)|^\alpha\leq \gamma, \quad\mbox{ for $n$ sufficiently large}.
\end{equation}
For every $\ell=1,\dots,2^{kd}$, we consider the cone over $(\mu_n-\nu_n)\trace Q^\ell$ of vertex $x_\ell$ as in \eqref{defn:cone}
$$C^\ell:=x_\ell\cone \big((\mu_n-\nu_n)\trace Q^\ell\big).$$ 
Its boundary is given by $(\mu_n-\nu_n)\trace Q^\ell + \sigma_n\trace Q^\ell$.
Denoting by $l$ the diameter of $Q$ and $C_1:=\sum_{\ell=1}^{2^{kd}}C^\ell$, we have
\begin{equation}\label{cono1}
\begin{split}
\MM(C_1)&\leq\sum_{\ell=1}^{2^{kd}}\MM(C^\ell)\overset{\eqref{stimecono}}{\leq} 2^{-k}l\sum_{\ell=1}^{2^{kd}}(\MM(\mu_n\trace Q^\ell)+\MM(\nu_n\trace Q^\ell))\\
&\overset{\eqref{griglia2}}{\leq} 2^{-k}l(\MM(\mu_n)+\MM(\nu_n))\leq  2^{-k+1}lC,
\end{split}
\end{equation}
and 
\begin{equation}\label{cono2}
\partial C_1=\mu_n-\nu_n + \sigma_n.
\end{equation}
Denote also $x$ the center of $Q$ and $C_2:= x\cone \sigma_n$. Again by \eqref{stimecono} and \eqref{massaalfacono}, since $\sum_{\ell=1}^{2^{kd}}\theta_\ell=0$ we have
\begin{equation}\label{cono3}
\partial C_2=\sigma_n \quad\mbox{ and }\quad \MM(C_2)\leq l\cdot \gamma.
\end{equation}
Combining \eqref{cono1}, \eqref{cono2}, and \eqref{cono3}, we deduce that
$$\partial(C_1-C_2)=\mu_n-\nu_n, \quad \mbox{ and } \quad \MM(C_1-C_2)\leq l(2^{-k+1}C+\gamma).$$
The conclusion follows from the arbitrariness of $k$ and $\gamma$.
\end{proof}

\subsection{Slicing}\label{s:slicing}
A fundamental tool for the proof of Theorem \ref{thm:main} is the notion of slicing of rectifiable 1-currents. Here we recall the definition and some fundamental properties. We refer the reader to \cite[Section 28]{SimonLN} for further details\footnote{As many classical references, \cite{SimonLN} considers only rectifiable currents with integer multiplicities. It is easy to check that every statement we refer to is valid also in the case of real multiplicities.}. 
\begin{definition}[(Slicing of 1-rectifiable currents)]
Let $T=[E,\tau,\theta] \in \mathbf{R}_1(\R^d)$ and let $f:\R^d \to \R$ be a Lipschitz function. For a.e. $t\in\R$ we define the slice of $T$ according to $f$ at $t$ to be the 0-rectifiable current
$$\langle T, f, t \rangle=[E_t,\tau_t,\theta_t],$$ 
where:
\begin{itemize}
\item $E_t=E\cap f^{-1}(t)$ and it is at most countable (hence 0-rectifiable) for a.e. $t$;
\item $\tau_t(x)=1$ if the scalar product $\nabla_Ef(x)\cdot\tau(x)$ is positive (where $\nabla_Ef$ denotes the tangential gradient); $\tau_t(x)=-1$ otherwise; 
\item $\theta_t=1_{E_t}\theta$.
\end{itemize}
\end{definition}

We will use the following characterization of the slices (see \cite[Lemma 28.5(2)]{SimonLN}). Let $T$ and $f$ as above. Then
\begin{equation}\label{e:def_slicing}
\langle T, f, t \rangle:= \partial (T \trace \{f < t\})-(\partial T)\trace \{f < t\},
\end{equation}
for a.e. $t \in (0,+\infty)$.

We conclude this short review with a simple consequence of the Coarea formula for rectifiable sets (see \cite[Lemma 28.5(1)]{SimonLN}). Let $T$ and $f$ as above, then
\begin{equation}\label{e:massa_slices}
\int_a^b\Mass (\langle T, f, t \rangle) dt\leq {\rm Lip}(f)\Mass (T\trace \{a<f<b\}).
\end{equation}

In the following, we choose $f:=d_x$, where $d_x(z):=\|z-x\|_\infty$.
\begin{lemma}[(Estimate of $\MM$ of suitable slices)]\label{slice}
	Let $x,y \in \R^d$, $r_0>0, \eta_0 \in (1,2)$, $\{T_n=[E_n,\tau_n,\theta_n]\}_{n\in \N} \subset \mathbf{R}_1(\R^d)$ with $\MM(T_n)\leq C.$
	Then there exists a set of positive measure $E \subseteq [r_0, \eta_0r_0]$
	 such that for every $r\in E$ there exist infinitely many $n\in \N$ satisfying
	\begin{equation}\label{eqn:slicing_control}
\MM(\langle T_n, d_{x}, r \rangle )+\MM(\langle T_n, d_{y}, r \rangle ) \leq 4\frac{ \MM(T_n)}{(\eta_0-1)r_0},\quad\mbox{ for $j=1,\dots,m$.}
	\end{equation}
\end{lemma}
\begin{proof}
    For every $n\in \N$ we define the set 
	$$F_n:=\Big\{r\in (r_0, \eta_0r_0): \MM(\langle T_n, d_{x}, r \rangle)+\MM(\langle T_n, d_{y}, r \rangle) \leq \frac{4 \MM(T_n)}{(\eta_0-1)r_0} \Big\}.
	$$	
	We apply Chebyshev inequality and \eqref{e:massa_slices} to the $1$-rectifiable current $\tilde T_n =[E_n,\tau_n,\theta_n^\alpha]$
	to obtain
	$$\Haus^1((r_0, \eta_0r_0)\setminus F_n)  \frac{4 \MM(T_n)}{(\eta_0-1)r_0}\leq  
	\int_{r_0}^{\eta_0r_0} \MM(\langle T_n, d_{x}, r \rangle)+\MM(\langle T_n, d_{y}, r \rangle) dr$$ 
	$$= \int_{r_0}^{\eta_0r_0} \Mass(\langle \tilde T_n, d_{x}, r \rangle)+\Mass(\langle \tilde T_n, d_{y}, r \rangle) dr \leq 2\Mass (\tilde T_n) = 2\MM(T_n).$$
	We deduce that $\Haus^1(F_n) \geq (\eta_0-1)r_0/2$. By Fatou's lemma
	$$\frac{(\eta_0-1)r_0}{2}\leq \limsup_{n\to \infty} \int_{r_0}^{\eta_0r_0} 1_{F_n}(r) \, dr \leq  \int_{r_0}^{\eta_0r_0} \limsup_{n\to \infty} 1_{F_n}(r) \, dr ,$$
	hence there exists a set of positive measure of radii where $ \limsup_{n\to \infty} 1_{F_n}(r) =1$. Any $r$ in this set satisfies \eqref{eqn:slicing_control} (for a possibly $r$-dependent family of indices $n$).
\end{proof}

\subsection{Improved lower semi-continuity} 
Given $\{x_1, \dots, x_N\}\in\R^d$ we consider a sequence of sets $\{G_k\}_{k\in\N}$ with the following property. For every $k$ there are closed disjoint cubes $Q^k_1,\dots,Q^k_N$ of diameters $\rho^k_1, \dots, \rho^k_N$ such that $\rho^k_j\to 0$ for every $j=1,\dots,N$, as $k\to 0$, $x_j\subset Q^k_j$ for $j=1,\dots,N$ and moreover $Q^k_j\supset Q^h_j$, for every $h>k$, for every $j=1,\dots,N$. Define 
\begin{equation}\label{e:gicappa}
G_k=\R^d\setminus\bigcup_{j=1}^NQ^k_j.
\end{equation}

\begin{lemma}\label{l:andrea}
Let $\{G_k\}_{k \in \N}$ be as in \eqref{e:gicappa} and let $\{T_n\}_{n \in \N} \subset \mathbf{R}_{1}(\R^d)$  and $T \in \mathbf{R}_{1}(\R^d)$ such that 
\begin{equation}\label{e:andrea1}
\lim_{n \to +\infty} \Flat (T_n- T) = 0.
\end{equation}
Then there exists a subsequence $\{T_{n_k}\}$ and a sequence of open sets $G'_k\subset G_k$ such that 
\begin{equation}\label{e:andrea2}
\lim_{k \to +\infty}\Flat (T_{n_k}\trace G'_k-T)=0.
\end{equation}
\end{lemma}
\begin{proof}
For every $n\in\N$, let $\varepsilon_n:=\Flat(T_n-T)$. By assumption $\varepsilon_n\to 0$ as $n\to +\infty$. For every $k\in \N$, let $\rho_k>0$ be such that $\rho_k\to 0$ as $k\to\infty$ and ${\rm dist}(Q^k_i,Q^k_j)\geq 2\rho_k$, for every $1\leq i< j\leq N$. By definition of flat distance, for every $n\in\N$ there exist $R_n,S_n$ such that $T_n-T=R_n+\partial S_n$ and $\Mass (R_n)+\Mass (S_n)\leq 2\varepsilon_n$. Choose $n_k$ such that $2\varepsilon_{n_k}\leq \rho_k^2$. By \eqref{e:massa_slices}, for every $k$ and for every $j=1,\dots,N$, there exists $0<r^k_j<\rho_k$ such that, denoting $d^k_j(x):={\rm dist}(x, Q^k_j)$, we have
\begin{equation}\label{e:slic1}
\sum_{j=1}^N\Mass(\langle S_{n_k}, d^k_j, r^k_j\rangle)\leq \rho_k^{-1}\varepsilon_{n_k}\leq\rho_k.
\end{equation}
Denote $G'_k:=\R^d\setminus \cup_{j=1}^N \overline{B(Q^k_j, r^k_j)}$. Obviously $G'_k\subset G_k$ for every $k$. 
Moreover, since $$T_{n_k}\trace G'_k-T=T_{n_k}-T - T_{n_k}\trace (\R^d\setminus G'_k),$$
then in order to prove \eqref{e:andrea2} it is sufficient to prove that
$$\lim_{k\to\infty}\Flat(T_{n_k}\trace (\R^d\setminus G'_k))=0.$$
Observe firstly that $\Mass(T\trace (\R^d\setminus G'_k))\leq \Mass(T\trace (\cup_{j=1}^N \overline{B(Q^k_j, \rho_k)}))\to 0$, as $k\to\infty$ because $\cup_{j=1}^N \overline{B(Q^k_j, \rho_k)}$ monotonically converges to the complementary of the $\Haus^1$-null set $\{x_1, \dots, x_N\}$, hence 
$$\lim_{k\to\infty}\Flat(T\trace (\R^d\setminus G'_k))=0.$$
Therefore, it suffices to show that
$$\lim_{k\to\infty}\Flat((T_{n_k}-T)\trace (\R^d\setminus G'_k))=0.$$
We can write, denoting $\langle S_{n_k},\partial G'_k\rangle:=\sum_{j=1}^N\langle S_{n_k}, d^k_j, r^k_j\rangle$,
\begin{equation}
\begin{split}
(T_{n_k}-T)\trace (\R^d\setminus G'_k)&= R_{n_k}\trace (\R^d\setminus G'_k)+\partial S_{n_k}\trace (\R^d\setminus G'_k)\\
&\stackrel{\eqref{e:def_slicing}}{=} R_{n_k}\trace (\R^d\setminus G'_k)+\partial (S_{n_k}\trace (\R^d\setminus G'_k))+\langle S_{n_k},\partial G'_k\rangle.
\end{split}
\end{equation}
Hence, denoting $R'_k:=R_{n_k}\trace (\R^d\setminus G'_k)+\langle S_{n_k},\partial G'_k\rangle$ and $S'_k:=S_{n_k}\trace (\R^d\setminus G'_k)$, we have $(T_{n_k}-T)\trace (\R^d\setminus G'_k)=R'_k+\partial S'_k$, and, by \eqref{e:andrea1}, $\Mass(R'_k)+\Mass(S'_k)\leq\rho_k+2\varepsilon_{n_k}$,  which tends to 0 as $k\to\infty$.
\end{proof}

We improve \cite[Lemma 4.10]{CDRM1} as follows:
\begin{lemma}[(Semi-continuity with lower bound on the density)]\label{second}
Let $T \in \mathbf{R}_{1}(\R^d)$. For every $\Delta>0$, there exists $\delta_{T,\Delta}>0$ satisfying the following property.
Let $\{G_k\}_{k \in \N}$ be as in \eqref{e:gicappa} and let $\{T_n\}_{n \in \N} \subset \mathbf{R}_{1}(\R^d)$ such that $T_n=[E_n,\tau_n,\theta_n]$ and 
\begin{equation}\label{ass1}
\mathbb M^\alpha(T_n) + \MM(T) \leq C, \quad \text{and} \quad \lim_{n \to +\infty} \Flat (T_n- T) = 0.
\end{equation}
Then there exists $\bar k \in \N$ such that for any $k \geq \bar k$ and for infinitely many $n$ (possibly depending on $k$)
\begin{equation}\label{concl}
\mathbb{M}^\alpha\Big (T_{n}\trace \Big (G_{k} \cap \Big \{|\theta_{n}|>\Big(\frac{\delta_{T,\Delta}}{2C}\Big)^\frac{1}{1-\alpha}\Big\}\Big)\Big ) \geq  \mathbb{M}^\alpha(T)-\Delta.
\end{equation}
\end{lemma}
\begin{proof}
Given $\Delta>0$, let $\delta_{T,\Delta}>0$ be such that, by the lower semi-continuity of $\mathbb{M}^\alpha$ with respect to the flat convergence (as stated in \cite[Proposition 2.5]{flat-relax}),
\begin{equation}\label{semic}
\Flat (T- T') \leq \delta_{T,\Delta} \quad \Rightarrow \quad \mathbb{M}^\alpha(T) \leq  \mathbb{M}^\alpha(T')+\frac{\Delta}{2}.
\end{equation}
Let us denote $\e=(\sfrac{\delta_{T,\Delta}}{2C})^\frac{1}{1-\alpha}$. By contradiction, there exist increasing sequences $k_i$ and $m_i$ such that
 \begin{equation}\label{concl1}
\mathbb{M}^\alpha(T_{n}\trace (G_{k_i} \cap \{|\theta_{n}|> \e\})) <  \mathbb{M}^\alpha(T)-\Delta, \quad \forall n\geq m_i, \, \, \forall i \in \N.
\end{equation}
By Lemma \ref{l:andrea}, there exists a subsequence $\{T_{n_i}\}_{i \in \N}\subset \{T_{m_i}\}_{i \in \N}$ and a sequence of open sets $G'_{k_i}\subset G_{k_i}$ such that
\begin{equation}\label{ass2}
\Flat (T_{n_i}\trace G_{k_i}' - T) \leq \frac{\delta_{T,\Delta}}{2}, \qquad \forall i \in \N.
\end{equation}
Moreover, since $m_i$ is an increasing sequence, we deduce that $n_i \geq m_i$.

 By \eqref{ass1} it holds
	\begin{equation}
	\label{eqn:mass-to-0}
	\mathbb M(T_{n_i}\trace (G'_{k_i}\cap \{|\theta_{n_i}|\leq \e \}))<\e^{1-\alpha}\mathbb M^\alpha(T_{n_i}\trace (G'_{k_i} \cap \{|\theta_{n_i}|\leq \e\}))<C\e^{1-\alpha}.
	\end{equation}
	Hence, by \eqref{ass2} and \eqref{eqn:mass-to-0}, we compute
	\begin{equation}\label{contr}
	\begin{split}
	\mathbb{F}(T-T_{n_i}\trace (G'_{k_i} \cap \{|\theta_{n_i}|> \e\}))
	&\leq\mathbb{F}(T-T_{n_i}\trace G'_{k_i})+\mathbb{F}(T_{n_i}\trace G'_{k_i} - T_{n_i}\trace (G'_{k_i} \cap \{|\theta_{n_i}|> \e\}))\\
	& =\mathbb{F}(T-T_{n_i}\trace G'_{k_i})+\mathbb{F}(T_{n_i}\trace (G'_{k_i} \cap \{|\theta_{n_i}|\leq \e\})) 
	\\ & \overset{\eqref{ass2}}{\leq} \frac{\delta_{T,\Delta}}{2}+ \mathbb M(T_{n_i}\trace (G'_{k_i} \cap \{|\theta_{n_i}|\leq \e\}))\overset{\eqref{eqn:mass-to-0}}{\leq} \frac{\delta_{T,\Delta}}{2}+ C \e^{1-\alpha} \\
	&\leq \frac{\delta_{T,\Delta}}{2}+ \frac{\delta_{T,\Delta}}{2}=\delta_{T,\Delta}.
	\end{split}
	\end{equation}
	Combining \eqref{contr}, \eqref{concl1} and \eqref{semic}, for every $i \in \N$, we deduce the desired contradiction
	$$\mathbb M^\alpha(T) \overset{\eqref{semic}}{\leq} \mathbb M^\alpha(T_{n_i}\trace (G'_{k_i} \cap \{|\theta_{n_i}|> \e\})) + \frac{\Delta}{2}\leq \mathbb M^\alpha(T_{n_i}\trace (G_{k_i} \cap \{|\theta_{n_i}|> \e\})) + \frac{\Delta}{2} \overset{\eqref{concl1}}{<} \mathbb M^\alpha(T) - \frac{\Delta}{2}.$$ 
\end{proof}
\subsection{Currents with coefficients in $\R^M$}\label{s:g-currents}
A technical difficulty in the proof of Theorem \ref{thm:main} comes from the fact that the limit of a sequence of good decompositions (as in Definition \ref{defn:GD}) is not necessarily a good decomposition. More precisely, we need a lower semi-continuity type result, which 
heuristically keeps track in the limit of those Lagrangian trajectories which have opposite orientations and therefore they would cancel as classical currents. To this aim we require notions from the theory of currents with coefficients in groups. In particular we work in the normed group $G:=(\R^M, \|\cdot \|_{1})$ and we obtain in Lemma \ref{lemmasem} a stronger statement with respect to the usual lower semi-continuity of the $\alpha$-mass. 

For the purposes of this paper it is sufficient to regard a current $T$ on $\R^d$ with coefficients in $\R^M$ as an ordered $M$-tuple of classical currents on $\R^d$ (i.e. with real coefficients), henceforth called the \emph{components} of $T$, and denoted $T^1,\dots, T^M$.
In particular one can represent a rectifiable 1-current $T$ on $\R^d$ with coefficients in $\R^M$ as a triple $[E,\tau,\Theta]$, where $E$ is a 1-rectifiable set on $\R^d$, $\tau$ is an orientation of $E$ and $\Theta=(\theta_1,\dots,\theta_M):E\to\R^M$, with $\Theta\in L^1(\Haus^1\trace E;\R^M)$. The components of $T$ are the classical 1-rectifiable currents $T^j:=[E,\tau,\theta_j]$, for $j=1,\dots,M$. The space of 1-rectifiable currents on $\R^d$ with coefficients in $\R^M$ is denoted ${\bf R}_1^{\R^M}(\R^d)$. 
We refer the reader to \cite[Section 4]{MMST} for a more rigorous introduction. 

For every $\alpha\in (0,1)$ and for $T=[E,\tau,\Theta]\in{\bf R}_1^{\R^M}(\R^d)$ we define the quantity 
$$\Mass_{\R^M}^\alpha(T):=\int_{E}\|\Theta\|_{1}^\alpha d\Haus^1.$$
By \cite[Section 6]{White1999} this quantity is lower semi-continuous with respect to the standard notion of convergence in flat norm for currents with coefficients in groups, which by \cite[Section 4.6]{MMST} is equivalent to the joint convergence in flat norm of all components.


\begin{lemma}[(Lower semi-continuity without cancellations)]\label{lemmasem}
For every $n \in \N$, let $\{T_n^\ell\}_{\ell=1}^M$, $\{T^\ell\}_{\ell=1}^M \subset \mathbf{R}_1(\R^d)$ with $T_n^\ell=[E_{n,\ell}, \tau_{n,\ell}, \theta_{n,\ell}]$ and $T^\ell=[E_\ell, \tau_\ell, \theta_\ell]$. We assume that
\begin{equation}\label{gruppo0}
\lim_{n \to +\infty}\Flat(T_n^\ell-T^\ell)= 0, \quad \forall \ell=1, \dots, M,
\end{equation}
and
\begin{equation}\label{gruppo1}
\Mass(T_n)=\sum_{\ell=1}^M\Mass(T_n^\ell), \quad \mbox{ where } \quad T_n:=\sum_{\ell=1}^MT_n^\ell.
\end{equation}
We denote $E=\cup_{\ell=1}^ME_\ell$ and $\theta: x \in E \mapsto \sum_{\ell=1}^M|\theta_\ell(x)|$. Then
\begin{equation}\label{gruppo2}
\int_E\theta^\alpha d\Haus^1 \leq \liminf_{n \to \infty}\MM(T_n).
\end{equation}
\end{lemma}
\begin{proof}
We first observe that by \eqref{gruppo1}, for every $n \in \N$, there exists a unitary vector field $\tau_n$ on $E_n:=\cup_{\ell=1}^ME_{n,\ell}$ such that 
\begin{equation}\label{gruppo3}
T_n=[E_{n}, \tau_{n}, \theta_{n}],\quad \mbox{ where $\theta_n:=\sum_{\ell=1}^M|\theta_{n,\ell}|$.}
\end{equation}

For every $\ell=1,\dots, M$, we can associate to the classical current $T^\ell_n$ the current $S^\ell_n=[E_{n,\ell}, \tau_{n,\ell}, \theta_{n,\ell}{\bf{e}}_\ell]\in {\bf R}_1^{\R^M}(\R^d)$. Analogously we associate to the current $T^\ell$ the currents $S^\ell=[E_{\ell}, \tau_{\ell}, \theta_{\ell}{\bf{e}}_\ell]$. We define $S_n:=\sum_{\ell=1}^MS^\ell_n$ and $S:=\sum_{\ell=1}^MS^\ell$. In other words $S_n$ is the current with coefficients in $\R^M$ whose components are $T_n^1,\dots, T_n^M$, while $S$ has components $T^1,\dots, T^M$.
By \eqref{gruppo1}, we can compute
\begin{equation}\label{gruppo5}
\MM_{\R^M}(S_n)=\int_{E_n}\bigg(\sum_{\ell=1}^M|\theta_{n,\ell}|\bigg)^\alpha d \Haus^1\overset{\eqref{gruppo3}}{=}\MM(T_n).
\end{equation}
By the lower semi-continuity of $\MM_{\R^M}$, (see \cite[Section 6]{White1999}), we deduce that
$$\int_E\theta^\alpha d\Haus^1 = \int_E\|(\theta_1,\dots,\theta_M)\|_1^\alpha d\Haus^1=\MM_{\R^M}(S)\stackrel{\eqref{gruppo0}}{\leq} \liminf_{n \to \infty}\MM_{\R^M}(S_n)\overset{\eqref{gruppo5}}{=}\liminf_{n \to \infty}\MM(T_n).\qedhere$$
\end{proof}

\section{Proof of Theorem \ref{thm:main}}\label{pro}
Up to a simple scaling argument (detailed at the beginning of the proof of \cite[Theorem 1.2]{CDRM1}), we can assume without loss of generality that $\Mass(\mu_n^\pm)=\Mass(\mu^\pm)=1$. By contradiction, we assume that there exists a (non-relabelled) subsequence $\{T_n\}_{n \in \N}$ and a traffic path $T \in \TP(\mu^-,\mu^+)$ such that $\Flat (T_n -T) \to 0$ and $T$ is not optimal. We consider $T_{opt}\in \OTP(\mu^-,\mu^+)$ and denote 
\begin{equation}\label{gap}
\Delta:=\MM(T)-\MM(T_{opt})>0.
\end{equation}
Let $\delta_{\Delta/4}>0$ be defined as in Lemma \ref{second} with respect to $\Delta/4$ and $T$, denote 
\begin{equation}\label{e:definCmod}
C:=\sup_{n\in \N}\MM(T_n)
\end{equation}
and fix
\begin{equation}\label{eps}
\e:=\min\left\{\frac{\Delta}{16},\left(\frac{\Delta}{8C}\right)^\frac{2}{\alpha}, \left(\frac{\delta_{\Delta/4}}{2C}\right)^\frac{2}{1-\alpha} \right\},
\end{equation}

\bigskip
{\it Step 1: Partitioning Smirnov curves of $T_n$ according to their initial and final points.} 

Since $T_n$ are optimal traffic paths, by Theorem~\ref{t:propr_good_dec} we can find for every $n \in \N$ a good decomposition (see Definition~\ref{defn:GD})
$$T_n=\int_{\Lip}I_\gamma dP_n(\gamma), \quad \mbox{ with } P_n \in \Prob(\Lip), \mbox{ supported on curves parametrized by arc length}.$$

Applying Lemma \ref{lemmagriglia}, we can find a cube $Q$ containing $X$, such that 
\begin{equation}
\label{eqn:grid-no-meas}
\mu^\pm(\Sk(Q,k))=\mu^\pm_n(\Sk(Q,k))=0, \quad \mbox{ for all } (k,n)\in\N^2.
\end{equation}
Without loss of generality we will assume that the edge length of $Q$ is 2, so that for every $Q^i\in\Lambda(Q,k)$ the distance between the center of $Q^i$ and $\partial Q^i$ is $2^{-k}$.
For every $k \in \N$, we consider $\Lambda(Q,k):=\{Q^\ell\}_{\ell=1}^{2^{kd}}$. Moreover, denoting $J_k:=\{1,\dots,2^{kd}\}^2$ for every $n \in \N$ and every $(i,j)\in J_k$, we define
\begin{equation}
\label{eqn:Pn}
T_n^{ij}:=\int_{\Lip(Q^i,Q^j)} I_\gamma d P_n(\gamma),
\end{equation}
which represents the portion of $T_n$ associated to the paths which begin in $Q^i$ and end in $Q^j$.

Notice that $T_n^{ij}$ depends implicitly on $k$; we will not explicit this dependence in the proof, apart from the steps 8 and 9 where the dependence on $k$ for the construction is more relevant.
By Theorem \ref{t:propr_good_dec}(3), we observe that \eqref{eqn:Pn} is a good decomposition. In particular, for every  $n\in \N$, denoting $T_n=[E_n,\tau_n,\theta_n]$, we have that $T_n^{ij}$ can be represented as $T_n^{ij}=[E_n,\tau_n,\theta_n^{ij}]$, with $\theta_n^{ij}(x)\cdot\theta_n(x)\geq 0$ and
\begin{equation}
\label{eqn:theta-n-ij}
 |\theta_n^{ij}(x)| \leq \min\{|\theta_n(x)|, P_n(\Lip(Q^i,Q^j))\}, \qquad \mbox{for $\Haus^1$-a.e $x\in E_n$}.
\end{equation}

\bigskip
{\it Step 2: Lagrangian description of $T$ and partition of the associated trajectories.} 

By Theorem~\ref{t:propr_good_dec}(2), $|\theta_n|\leq 1$ for $\Haus^1$-a.e. $x$. By \eqref{hp:energy-bound}, since $T_n$ are optimal, we deduce the following tightness condition for $P_n$:
\begin{equation}\label{e:definC}
\sup_{n\in \N}\int_{\Lip}{\text{length}}(\gamma) dP_n(\gamma)\overset{\eqref{eqn:buona-dec-mass-T}}{=}\sup_{n\in \N}\Mass(T_n)\overset{|\theta_n|\leq 1}{\leq} C <\infty.
\end{equation}
By \cite[Theorem 3.28]{BCM}, up to a further (non-relabelled) subsequence,  $P_n \wstarto P \in\Prob(\Lip) $. By \cite[Lemma 3.21]{BCM} $P$ is supported on eventually constant curves, and by Remark~\ref{rmk:limit} 
\begin{equation}\label{rep1}
T=\int_{\Lip}I_\gamma dP(\gamma).
\end{equation}
Notice that in general \eqref{rep1} could fail to be a good decomposition of $T$ in the sense of Definition~\ref{defn:GD}.

Analogously to \eqref{eqn:Pn}, one can define the portion of $T$ associated to the paths which begin in $Q^i$ and end in $Q^j$, as
$$ T^{ij}:=\int_{\Lip(Q^i,Q^j)} I_\gamma d P(\gamma).$$
Again we recall that the latter may fail to be a good decomposition.
By Theorem~\ref{t:propr_good_dec}(1) applied to $T^{ij}_n$ and $P_n$, we deduce that 
$$\partial_- T^{ij}_n = (e_0)_\# (P_n\trace \Lip(Q^i,Q^j)) \quad \text{ and } \quad \partial_+ T^{ij}_n = (e_\infty)_\# (P_n\trace \Lip(Q^i,Q^j)),$$
where $e_0: \gamma \in \Lip \mapsto \gamma(0)$ and $e_\infty: \gamma \in \Lip \mapsto \gamma(\infty)$.
 Passing to the limit in $n$, we deduce that
\begin{equation}
\label{Tij-boundary}
\partial_- T^{ij} 
= \int_{\Lip(Q^i,Q^j)}\delta_{\gamma(0)} d  P (\gamma)
\qquad \mbox{and} \qquad
\partial_+ T^{ij}  
= \int_{\Lip(Q^i,Q^j)}\delta_{\gamma(\infty)} d  P (\gamma).
\end{equation}
For every $(i,j)\in J_k$ we remark that $T_n^{ij} \wto T^{ij}$. Since $\Mass(T_n^{ij}) \leq 1$ and $\Mass(\partial T_n^{ij}) \leq 1$, by Remark \ref{rmk_andrea}, we have
\begin{equation}
\label{eqn:conv-Tij}
\lim_{n}\Flat(T_n^{ij}-T^{ij})=0.
\end{equation}
Indeed, $P_n\trace \Lip(Q^i,Q^j) \wstarto P\trace \Lip(Q^i,Q^j)$, because they are obtained localizing the weakly-$*$ converging sequence $P_n \wstarto P$ to the set $\Lip(Q^i,Q^j)$,
whose boundary has $0$ $P$-measure by \eqref{eqn:grid-no-meas}:
\begin{equation*}
\begin{split}
 P \big(\partial \Lip(Q^i,Q^j)\big) &\leq P \big(\{ \gamma\in  \Lip: \gamma(0) \in \partial Q^i\}\big) +  P \big(\{ \gamma\in  \Lip: \gamma(\infty) \in \partial Q^j \}\big) 
\\
&= \mu^- (\partial Q^i) + \mu^+(\partial Q^j) = 0.
\end{split}
\end{equation*}

\bigskip
{\it Step 3: Isolating ``bad'' cubes containing most of the atomic part of $\mu^\pm$.} 

In the following, given a measure $\nu \in \Mpiu(X)$, we denote by $\nu_a$ its atomic part, i.e. the only measure such that $\nu_a \leq \nu$, $\nu_a$ is supported on a countable set and $(\nu-\nu_a)(\{x\})=0$ for every $x \in X$.
Since $\mu^\pm$ are finite measures, there exists $N \in \N$, such that the sum of their atomic parts can be written as 
\begin{equation}
\label{eqn:pocamassaneicubi1}
\mu_a^++\mu_a^-=\bigg(\sum_{h=1}^Nc_h\delta_{x_h}\bigg)+\mu_r, \quad \text{with} \quad \mathbb{M}(\mu_r)<\eee,
\end{equation}
for some $c_1, \dots, c_N \in \R$ and $N$ distinct points $x_1, \dots, x_N \in X$, (we are implicitly assuming that the two addenda in the RHS of \eqref{eqn:pocamassaneicubi1} are mutually singular).  
We observe that, for every $k \in \N$, the set $\{x_h: h=1,\dots,N\}$ is contained in at most $N$ cubes of $\Lambda(Q,k)$.  By \eqref{eqn:grid-no-meas}, and since $\mu^+,\mu^-$ are mutually singular, there exists $k_0$ such that, for every $k \geq k_0$, all these cubes are disjoint (hence their mutual distances is larger or equal than the edge length of each cube, i.e. $2^{-k+1}$) and contain a single Dirac delta. For every $k \in \N$, up to reordering, we denote these cubes by $\{Q^h: h=1\dots,N\}$. Again, we do not explicit the dependence of these cubes on $k$, but we observe that their number $N$ does not depend on $k$.

We recall that $\sfrac{5}{4}Q^h$ is the concentric cube to $Q^h$, enlarged by the factor $\sfrac{5}{4}$, so that the cubes $\sfrac{5}{4}Q^h$ remain disjoint; we denote 
\begin{equation}\label{defB}
B_k=\bigcup_{h=1}^N \frac{5}{4}Q^h \qquad \text{ and } \qquad G_k:=B^c_k,
\end{equation} 
Since the sequence $B_k$ converges monotonically decreasing to the finite set $\{x_h: h=1\dots,N\}$, there exists $k_1\geq k_0$  such that, for every $k \geq k_1$,
\begin{equation}\label{cubi cattivi}
\MM(T\trace B_k) = \int_{B_k\cap E} |\theta|^\alpha\, d\Haus^1 \leq \e.
\end{equation}

\bigskip
{\it Step 4: Multiplicity estimate for the pieces of $T_n$ which do not connect bad cubes.} 

Since $\mu^\pm_d:=\mu^\pm - \mu^\pm_a$ has trivial atomic part, then there exists $k_2 \geq k_1$  such that, for every $k\geq k_2$
\begin{equation}
\label{eqn:pocamassaneicubi0}
\max\{ \mu^+_d(Q^i), \mu^-_d(Q^i)\} <\eee, \qquad \text{for all } Q^i \in \Lambda(Q,k).
\end{equation}
Then, by \eqref{eqn:pocamassaneicubi1}, for every $k\geq k_2$
$$
\max\{ \mu^+(Q^i), \mu^-(Q^i)\} <2\eee, \qquad \text{for all } Q^i \in \Lambda(Q,k)\setminus \{Q^h: h=1\dots,N\} .
$$
Hence, by \eqref{hp:supp-n-convergence} and \eqref{eqn:grid-no-meas}, for every $k \geq k_2$ there exists $n_0= n_0(k)$ 
such that, for every $n \geq n_0$
\begin{equation}
\label{eqn:pocamassaneicubi}
\max\{ \mu^+_n(Q^i), \mu^-_n(Q^i)\} <3\eee, \qquad \text{for all } Q^i \in \Lambda(Q,k)\setminus \{Q^h: h=1\dots,N\} .
\end{equation}
Since $\mu^+$ and $\mu^-$ are mutually singular by assumption, and since each cube in $\{Q^h: h=1\dots,N\}$ contains at most $1$ of the $N$ points $x_1,\dots,x_N$, then by \eqref{eqn:pocamassaneicubi1} and \eqref{eqn:pocamassaneicubi0}, for every $k \geq k_2$
\begin{equation}
\label{eqn:pocamassaneicubi2}
\min\{ \mu^+(Q^i), \mu^-(Q^i)\} <2\eee \qquad \mbox{for all }  Q^i \in \Lambda(Q,k).
\end{equation}
Hence, for every $k \geq k_2$, there exists $n_1=n_1(k)\geq n_0(k)$  
such that for every $n \geq n_1$
\begin{equation}
\label{eqn:pocamassaneicubi3}
\quad\min\{ \mu^+_n(Q^i), \mu^-_n(Q^i)\} <3\eee, \qquad \text{for all } Q^i \in \Lambda(Q,k).
\end{equation}

Using Theorem~\ref{t:propr_good_dec} (1,3) applied to $T^{ij}_n$, we deduce from \eqref{eqn:pocamassaneicubi3} that, for every couple of cubes $Q^i,Q^j$ such that either $Q^i$ or $Q^j$ belong to $ \Lambda(Q,k)\setminus \{Q^h: h=1\dots,N\}$, for every $k\geq k_2$ and for every $n\geq n_1$,
\begin{equation}\label{small density}
\begin{split}
|\theta_{T^{ij}_n}(x)| \leq P_n(\Lip(Q^i,Q^j)) \leq \min\{ \partial_- T_n(Q^i),  \partial_+ T_n(Q^j)\} = \min\{ \mu^-_n(Q^i),   \mu^+_n(Q^j)\} \leq 3\eee,
\end{split}
\end{equation}
for $\Haus^1$-a.e. $x\in E_n$.


\bigskip

{\em Step 5: Choice of slightly enlarged cubes to have a control on the slices.} 
%

In the following we use the short notation $S_n^{ij}(\rho)$ and $S^{ij}(\rho)$ to denote respectively
$$\langle T_n^{ij},d_{x_i},\rho\rangle + \langle T_n^{ij},d_{x_j},\rho\rangle \quad\mbox{and}\quad \langle T^{ij},d_{x_i},\rho\rangle + \langle T^{ij},d_{x_j},\rho\rangle,$$
where $x_i$ denotes the center of the cube $Q^i$ and $d_x$ is defined in \S \ref{s:slicing}. 

For every $k\in\N$, and for a given pair $(i,j)\in J_k$, applying Lemma \ref{slice}, we get that, up to a (non relabelled) subsequence $\{T_n\}_{n\in\N}$, there exists a set of positive measure of radii $\rho_k^{ij}\in (2^{-k},\tfrac{5}{4}2^{-k})$ such that 
\begin{equation}\label{stima2-firstversion} 
\MM(S_n^{ij}(\rho_k^{ij})) \leq 4 \frac{ \MM(T_n^{ij})}{2^{-k-2}}\leq 2^{k+4}\MM(T_n)\stackrel{\eqref{e:definCmod}}{\leq} 2^{k+4}C,\quad \mbox{for every $n\in\N$},
\end{equation}
where the second inequality follows form Theorem \ref{t:propr_good_dec} (3). Since by \eqref{eqn:conv-Tij} and \eqref{e:def_slicing} for almost every radius $\rho_k^{ij}$
\begin{equation}\label{stima2conv}
\lim_{n\to \infty}\Flat(S^{ij}(\rho_k^{ij}) - S_n^{ij}(\rho_k^{ij}))=0,
\end{equation}
by lower semi-continuity of $\MM$ with respect to the flat convergence we deduce that
\begin{equation}\label{stima_massa_alfa_slice_T}
\MM(S^{ij}(\rho_k^{ij})) \leq 2^{k+4}C.
\end{equation}
Since for every $k\in\N$ the number of possible pairs $(i,j)$ is finite, up to choosing iteratively a (non relabelled) subsequence $\{T_n\}_{n\in\N}$, we can assume that estimates \eqref{stima2conv} and \eqref{stima_massa_alfa_slice_T} hold for every $(i,j) \in J_k$.

We observe that $\partial T_{n}^{ij}\trace (\rho_k^{ij}Q^i\cup \rho_k^{ij} Q^j)=\partial T_{n}^{ij}$ and analogously $\partial T^{ij}\trace (\rho_k^{ij}Q^i\cup \rho_k^{ij} Q^j)=\partial T^{ij}$, which combined with \eqref{e:def_slicing} gives respectively
\begin{equation}\label{e:defSnij}
S_n^{ij}(\rho_k^{ij})
=\partial (T_{n}^{ij}\trace (\rho_k^{ij}Q^i\cup \rho_k^{ij}Q^j))- \partial T_{n}^{ij},
\end{equation}
and
\begin{equation}\label{e:defSij}
 S^{ij}(\rho_k^{ij})=\partial (T^{ij}\trace (\rho_k^{ij}Q^i\cup \rho_k^{ij}Q^j))- \partial T^{ij}\trace (\rho_k^{ij}Q^i\cup \rho_k^{ij} Q^j).
\end{equation}
Consequently, we deduce respectively that
\begin{equation}
\label{eqn:sij}
[S^{ij}_n(\rho_k^{ij})](\R^d) = 0, \quad \text{and} \quad [S^{ij}(\rho_k^{ij})](\R^d) = 0.
\end{equation}
We denote
$$S:=\sum_{(i,j)\in J_k}S^{ij}(\rho_k^{ij}) \quad \text{and} \quad S_n:=\sum_{(i,j)\in J_k}S^{ij}_n(\rho_k^{ij}).$$
{\em Step 6: Transport between $\partial T$ and the corresponding slices $S$.} 

We define 
\begin{equation}\label{def1}
T_{n,\3}^{ij}:=T_{n}^{ij}\trace (\rho_k^{ij}Q^i\cup \rho_k^{ij}Q^j), \quad   T_{n,\3}:=\sum_{(i,j)\in J_k}  T_{n,\3}^{ij} 
.
\end{equation}
We remark that, by Theorem \ref{t:propr_good_dec}(2), for $\Haus^1$-a.e. $x$
\begin{equation}\label{stimadens}
\begin{split}
|\theta_{T_{n,\3}}(x)|&= \bigg|\sum_{(i,j)\in J_k}  \theta_{T_{n,\3}^{ij}}(x)\bigg| \leq \sum_{(i,j)\in J_k}  \theta_{T_{n}^{ij}}(x) \\
&\leq \sum_{(i,j)\in J_k}  P_{n}(\{\gamma \in \Lip(Q^i,Q^j): x \in {\rm{Im}}\gamma \}) \leq P_{n}(\{\gamma \in \Lip: x \in {\rm{Im}}\gamma \}) = |\theta_{T_{n}}(x)|.
\end{split}
\end{equation}
We observe that, since $\partial T_{n}^{ij}= \partial T_{n}^{ij}\trace (Q^i\cup Q^j)$, by \eqref{e:def_slicing}
\begin{equation}\label{eqn:tn3-boundary}
\partial T_{n,\3}=\sum_{(i,j)\in J_k}\partial T_{n,\3}^{ij}= \sum_{(i,j)\in J_k}\partial T_{n}^{ij}\trace (\rho_k^{ij}Q^i\cup \rho_k^{ij} Q^j)+ S_n^{ij}(\rho_k^{ij})=\sum_{(i,j)\in J_k}\partial T_{n}^{ij}+ S_n^{ij}(\rho_k^{ij})=\partial T_{n}+ S_n.
\end{equation}
Analogously one can define $T^{ij}_\3$ and $T_\3$ as
$$ 
T_{\3}^{ij}:=T^{ij}\trace (\rho_k^{ij}Q^i\cup \rho_k^{ij}Q^j), \quad   T_{\3}:=\sum_{(i,j)\in J_k}  T_{\3}^{ij}.
$$ 
We have
\begin{align}\nonumber
\partial T_\3&= \sum_{(i,j)\in J_k} \partial T_{\3}^{ij}= \sum_{(i,j)\in J_k} \partial (T^{ij}\trace (\rho_k^{ij}Q^i\cup \rho_k^{ij}Q^j))\stackrel{\eqref{e:defSij}}{=} \sum_{(i,j)\in J_k}  (\partial T^{ij}\trace (\rho_k^{ij}Q^i\cup \rho_k^{ij}Q^j)+ S^{ij}(\rho_k^{ij}))\\
&=\sum_{(i,j)\in J_k}  (\partial T^{ij}\trace (Q^i\cup Q^j))+S=(\mu^+-\mu^-)+S=\partial T_{opt}+S.\label{eqn:t3-boundary}
\end{align}

\bigskip
{\em Step 7: Connection of the slices of $T_n$ and $T$. } 

 We define 
$$\sigma_n:=\sum_{(i,j)\in J_k}  (S_n^{ij}(\rho_k^{ij}))_++\sum_{(i,j)\in J_k} ( S^{ij}(\rho_k^{ij}))_- \quad \mbox{and} \quad \nu_n:= \sum_{(i,j)\in J_k} ( S^{ij}(\rho_k^{ij}))_++\sum_{(i,j)\in J_k}  (S_n^{ij}(\rho_k^{ij}))_-.$$
We observe that
$$\sigma_n- \nu_n\wto 0, \quad \mbox{and} \quad \Mass(\sigma_n)=\Mass(\nu_n). $$ 
Indeed, the weak-$*$ convergence holds because, by \eqref{stima2conv}, we get
$$\sum_{(i,j)\in J_k}  (S_n^{ij}(\rho_k^{ij}))_+ \wto \sum_{(i,j)\in J_k}  (S^{ij}(\rho_k^{ij}))_+ 
\quad \mbox{and} \quad \sum_{(i,j)\in J_k}  (S_n^{ij}(\rho_k^{ij}))_-\wto  \sum_{(i,j)\in J_k} ( S^{ij}(\rho_k^{ij}))_-
.$$
By \eqref{eqn:sij}, we deduce that
$$\Mass (\nu_n)-\Mass(\sigma_n)=\sum_{(i,j)\in J_k}  (S^{ij}(\rho_k^{ij}))_+(\R^d) - 
(S^{ij}(\rho_k^{ij}))_-(\R^d)+
  (S_n^{ij}(\rho_k^{ij}))_-(\R^d)-
 (S_n^{ij}(\rho_k^{ij}))_+(\R^d)=0.$$
Moreover, thanks to \eqref{stima2-firstversion} and \eqref{stima_massa_alfa_slice_T} we have that
$$\MM(\sigma_n)+\MM(\nu_n)\leq 2^{2kd+k+6}C.$$
Applying Lemma \ref{high_multiplicity}, for every $k \geq k_2$, there exists $n_2=n_2(k)\geq n_1(k)$  
 such that for every $n\geq n_2$ there exists a transport $T_{n,conn}$ such that
\begin{equation}\label{connessionebord}
\partial T_{n,conn}= \nu_n- \sigma_n=S-S_n, \quad \mbox{and} \quad \MM(T_{n,conn})<\e.
\end{equation}

\bigskip
{\it Step 8: Improved semi-continuity of the energy to bound a modified density of $T$ which neglects cancellations among different partitions.} 

In this step we will label the dependence of $T^{ij}$ and $T_n^{ij}$ from $k$ explicitly, with the notation $T^{ij}_k$ and $T^{ij}_{n,k}$. In particular we write $T^{ij}_k=[E_k^{ij}, \tau_k^{ij}, \theta_k^{ij}]$. Let us consider the rectifiable set $E=\cup_{k\in \N} \cup_{i,j}E_k^{ij}$ and $\bar \theta_k=\sum_{ij}|\theta_k^{ij}|$. We claim that for $\Haus^1$-a.e. $x\in E$, the sequence $\bar \theta_k(x)$ is non-decreasing in $k$ and that, setting  $\bar \theta=\sup_{k\in \N}\bar \theta_k$, we have
\begin{equation}\label{bound1}
\int_{E}\bar \theta^\alpha d\Haus^1 \leq C.
\end{equation}

To prove this claim, we define the positive measures $\nu^{ij}_k:=|\theta_k^{ij}| \Haus^1\trace E_k^{ij}\in \Mpiu(\R^d)$ associated to $T^{ij}_k$ and the measure $\nu_k:=\sum_{ij}\nu^{ij}_k=\bar\theta_k \Haus^1\trace E$. By the good decomposition of $T_n$, we deduce that 
\begin{equation}\label{bound0}
\Mass(T_n)=\sum_{ij}\Mass(T^{ij}_{n,k}).
\end{equation}
By \eqref{eqn:conv-Tij} and \eqref{bound0}, we can then apply Lemma \ref{lemmasem} to the sequence $T^{ij}_{n,k}$ to deduce that for every fixed $k \in \N$
\begin{equation}\label{bound}
\int_{E}\bar \theta_k^\alpha d\Haus^1 \leq \liminf_{n \to \infty}\MM(T_n)\leq C.
\end{equation}
Furthermore, we observe that $\nu_k \leq \nu_{k+1}$ for every $k \in \N$. Indeed, 
$$\theta_k^{ij}=\sum_{s,t: Q^s\subset Q^i, Q^t\subset Q^j}\theta_{k+1}^{st},$$
where we intend that $Q^s, Q^t$ belong to $\Lambda(Q, k+1)$ and $Q^i, Q^j$ belong to $\Lambda (Q,k)$.
Therefore
$$\bar \theta_k=\sum_{ij}|\theta_k^{ij}|=\sum_{ij}\bigg|\sum_{s,t: Q^s\subset Q^i, Q^t\subset Q^j}\theta_{k+1}^{st}\bigg|\leq \sum_{ij}\sum_{st: Q^s\subset Q^i, Q^t\subset Q^j}|\theta_{k+1}^{st}|=\sum_{s,t}|\theta_{k+1}^{st}|= \bar \theta_{k+1}.$$

Consequently, the monotonicity together with the uniform bound in $k$ \eqref{bound}, yields \eqref{bound1}.

\bigskip
{\em Step 9: Energy estimate for $T_\3$.} 

We claim that there exist infinitely many indexes $\{k_h\}_{h\in\N}$ such that
\begin{equation}\label{eqn:en-T3}
\MM(T_\3)<\e.
\end{equation}
In the proof of this step we will trace the dependence of $T_\3$ from $k$ explicitly with the notation $T_\3^k$. We first observe that $\Mass(T_\3^k)\to 0$ as $k \to +\infty$.
To this aim, we denote by $\length(\gamma)$ the length of any curve $\gamma\in \Lip$. 
Since the function $\length$ is lower semi-continuous on $\Lip$ and $P_n$ converge weakly-$*$ as measures, by the good decomposition property \eqref{eqn:buona-dec-mass-T} of $T_n$, and since finally by Theorem~\ref{t:propr_good_dec}(2) the density of $T_n$ is bounded by $1= P_n(\Lip)$, we have
\begin{equation}
\label{eqn:length-int}
\int_{\Lip} \length(\gamma) \,dP \leq
\liminf_{n\to \infty} \int_{\Lip} \length(\gamma) \,dP_n = \liminf_{n\to \infty} \Mass(T_n) \leq \liminf_{n\to \infty} \MM(T_n).
\end{equation}
Hence we know that $\length(\gamma)\in L^1(P)$. Now we define
$$A_k(\gamma):= \bigcup \{\rho_k^{ij} Q^i \cup \rho_k^{ij} Q^j: \, Q^i,Q^j \in \Lambda(Q,k), \, \gamma(0) \in Q^i \, \text{and } \gamma(\infty)\in Q^j\},$$
and the function $\length_k:\Lip \to [0,+\infty)$ as
$$\length_k(\gamma):=\int_{0}^\infty|\dot{\gamma}|(t)\chi_{\{s : \gamma(s) \in A_k(\gamma)\}}(t)dt  = \Haus^1({\rm Im}\gamma \cap A_k(\gamma) ).$$ 
We can then estimate
$$\Mass(T_\3^k)\leq \int_{\Lip} \length_k(\gamma) dP(\gamma).$$
As observed above, the limit $P$ has the property that $\gamma$ is an eventually constant curve for $P$-a.e. $\gamma$. We consequently deduce that $\length_k(\gamma) \to 0$ for $P$-a.e. $\gamma \in \Lip$.
Moreover, $\length_k(\gamma)\leq \length(\gamma)$.
Since $\length\in L^1(P)$, by dominated convergence we deduce that 
\begin{equation}\label{mass}
\lim_{k\to \infty}\Mass(T_\3^k)\leq \lim_{k\to \infty}\int_{\Lip} \length_k(\gamma) dP(\gamma)= 0
.
\end{equation}
By \eqref{mass}, there exists a subsequence $\{k_h\}_{h\in\N}$ such that the density $\theta_{\3,{k_h}}$ of $T_{\3}^{k_h}$ satisfies $\theta_{\3,{k_h}}(x)\to 0$ as $h \to \infty$ for $\Haus^1$-a.e. $x\in E$. 
Moreover, thanks to \eqref{bound1}, we deduce that $|\theta_{\3,{k_h}}|^\alpha \leq \bar \theta_{k_h}^\alpha \leq \bar \theta^\alpha \in L^1(\Haus^1\trace E)$ (where the set $E$ and the multiplicities $\bar \theta_{k_h}$ and $\bar \theta$ have been defined in Step 8) and consequently, by dominated convergence, that
$$\MM(T_\3^{k_h})=\int_{E}|\theta_{\3,k_h}|^\alpha d\Haus^1 \to 0 \quad \mbox{as } h \to \infty,$$
which implies the claim in \eqref{eqn:en-T3}.

\bigskip
{\em Step 10: Construction of the energy competitor for $T_n$.}

In the rest of the proof we fix
\begin{equation*}
k\in\{k_h\}_{h\in\N} \text{ with } k\geq\max\{\bar{k}, k_2\}, \qquad \text{and} \qquad n\geq n_2(k),
\end{equation*}
 where $\bar{k}$ and $n$ are obtained in Lemma \ref{second}, with $\sfrac{\Delta}4$ in place of $\Delta$ and $\{G_k\}_{k \in \N}$, $\{T_n\}_{k \in \N}$ and $T$ being those used so far in the proof of Theorem \ref{thm:main}. We recall that $k_2$ was defined \eqref{eqn:pocamassaneicubi0}, $\{k_h\}$ in \eqref{eqn:en-T3}, and $n_2(k)$ in \eqref{connessionebord}.

 We deduce from \eqref{concl1} the following estimate
 \begin{equation}\label{conc2}
\mathbb{M}^\alpha\left (T_{n}\trace (G_{k} \cap \left \{\theta_{n}>\sqrt{\e}\right\})\right ) \overset{\eqref{eps}}{\geq} \mathbb{M}^\alpha\left (T_{n}\trace \left(G_{k} \cap \left \{\theta_{n}>\left(\frac{\delta_{\Delta/4}}{2C}\right)^\frac{1}{1-\alpha}\right\}\right )\right) \overset{\eqref{concl1}}{\geq}  \mathbb{M}^\alpha(T)-\frac{\Delta}{4}.
\end{equation}
In the first inequality we used that $\sqrt{\e} \leq \left(\sfrac{\delta_{\Delta/4}}{2C}\right)^\frac{1}{1-\alpha}$, by \eqref{eps}.
We define the following traffic path:
$$T_{n,comp}:= T_{n,conn}+T_{opt}-T_\3+T_{n,\3}.$$
This is a competitor for $T_n$, namely $\partial T_{n,comp}=\partial T_{n}$. Indeed, thanks to \eqref{eqn:t3-boundary}, \eqref{eqn:tn3-boundary}, and finally \eqref{connessionebord}, we compute
$$
\partial T_{n,comp}=\partial T_{n,conn}+ \partial T_{opt}-\partial T_\3+\partial T_{n,\3}= \partial T_{n,conn}
 -S+\partial T_{n}+ S_n \overset{\eqref{connessionebord}}{=}\partial T_{n}.
$$

\bigskip
{\em Step 11: Energy estimate and conclusion.}

To estimate the energy of the competitor $T_{n,comp}$ we first use the sub-additivity of $\MM$ and the smallness of the energy contributions of $T_{n,conn}$ and $T_\3$, in view of \eqref{connessionebord} and \eqref{eqn:en-T3}
.
We obtain that
$$
\MM(T_{n,comp})\leq \MM(T_{n,\3})+\MM(T_{n,conn})+\MM(T_{opt})+\MM(T_\3)\leq \MM(T_{n,\3})+\MM(T_{opt})+2\e,
$$
which, combined with \eqref{gap} and \eqref{conc2}, reads
\begin{equation}\label{eqn:en-est}
\MM(T_{n,comp})\overset{\eqref{gap}}{\leq}  \MM(T_{n,\3})+\MM(T)-\Delta+2\e \overset{\eqref{conc2}}{\leq} \MM(T_{n,\3})+\MM(T_n\trace (G_k\cap \{|\theta_n| >\sqrt{\e}\}))-\frac{3\Delta}{4}+2\e,
\end{equation}
Next, we call $\overline{T}_1:= T_{n,\3}$ and $\overline{T}_2:= T_n\trace (G_k\cap \{|\theta_n| >\sqrt{\e}\})$ and we estimate their densities.
 
We first observe that, by \eqref{defB}, it holds $G_k \cap (\rho_k^{ij}Q^i\cup \rho_k^{ij}Q^j)\subset B_k^c$. This implies that for every $x\in G_k \cap (\rho_k^{ij}Q^i\cup \rho_k^{ij}Q^j)$, either $Q^i$ or $Q^j$ belong to $ \Lambda(Q,k)\setminus \{Q^h: h=1\dots,N\}$.
Recalling the definition \eqref{def1} $T_{n,\3}^{ij}=T_{n}^{ij}\trace (\rho_k^{ij}Q^i\cup \rho_k^{ij}Q^j)$, applying \eqref{small density}, we can estimate the density of $\overline{T}_1$ as follows
\begin{equation}
\label{eqn:t1bounds-0}
|\theta_{\overline{T}_1}| \leq { \eee}\qquad \mbox{for $\Haus^1$-a.e. $x\in  G_k$}.
\end{equation}
Notice that \eqref{eqn:t1bounds-0} may no longer hold for $x\notin G_k$: indeed \eqref{small density} may fail if both $Q^i$ and $Q^j$ belong to $\{Q^h: h=1\dots,N\}$.


On the other side, the density of $\overline{T}_2$ satisfies
\begin{equation}
\label{eqn:t2bounds}
\sqrt {\eee} \leq |\theta_{\overline{T}_2}(x)| \leq |\theta_{T_n}(x)|,\qquad \mbox{for $\Haus^1$-a.e. $x\in  G_k\cap \{ |\theta_{\overline{T}_2}|>0\}$}.
\end{equation}
Combining the bounds \eqref{eqn:t1bounds-0} and \eqref{eqn:t2bounds}, we deduce that
\begin{equation}\label{ineq}
|\theta_{\overline{T}_1}|^\alpha+ |\theta_{\overline{T}_2}|^\alpha \leq \eee^\alpha+ |\theta_{\overline{T}_2}|^\alpha \leq (\eee^{\alpha/2} +1)|\theta_{\overline{T}_2}|^\alpha,\qquad \mbox{for $\Haus^1$-a.e. $x\in  G_k\cap \{ |\theta_{\overline{T}_2}|>0\}$}.
\end{equation}
We employ this inequality together with \eqref{stimadens} in the energy estimate 
\begin{equation*}
\begin{split}
\MM(\overline{T}_1)+ \MM(\overline{T}_2) &= \MM(\overline{T}_1 \trace (G_k^c \cup(G_k \cap \{ \theta_{\overline{T}_2}=0\})))+  \int_{G_k \cap \{ |\theta_{\overline{T}_2}|>0\}}|\theta_{\overline{T}_1}|^\alpha+|\theta_{\overline{T}_2}|^\alpha d\Haus^1
\\
&\overset{\eqref{stimadens},\eqref{ineq}}{\leq}\MM(T_n \trace (G_k^c \cup(G_k \cap \{ \theta_{\overline{T}_2}=0\})))+  (\eee^{\alpha/2} +1)\int_{G_k \cap \{ |\theta_{\overline{T}_2}|>0\}} |\theta_{\overline{T}_2}|^\alpha d\Haus^1
\\
&\overset{\eqref{eqn:t2bounds}}{\leq} \MM(T_n \trace (G_k^c \cup(G_k \cap \{ \theta_{\overline{T}_2}=0\})))+  (\eee^{\alpha/2} +1)\int_{G_k \cap \{ |\theta_{\overline{T}_2}|>0\}} \theta_{T_n}^\alpha d\Haus^1 
\\
&= \MM(T_n \trace (G_k^c \cup(G_k \cap \{ \theta_{\overline{T}_2}=0\})))+(\eee^{\alpha/2} +1)\MM(T_n\trace (G_k \cap \{ |\theta_{\overline{T}_2}|>0\}))\\
& = (\eee^{\alpha/2} +1)\MM(T_n).
\end{split}
\end{equation*}

We plug this estimate in \eqref{eqn:en-est} and we recall  that $\MM(T_n)\leq C$, so that
\begin{equation}\label{eqn:finita}
\begin{split}
\MM(T_{n,comp}) \leq (\eee^{\alpha/2} +1)\MM(T_n)-\frac{3\Delta}{4}+2\e \leq \MM(T_n)-\frac{3\Delta}{4}+2\e+C\eee^{\alpha/2} \overset{\eqref{eps}}{\leq} \MM(T_n)-\frac{\Delta}{2}.
\end{split}
\end{equation}
The estimate \eqref{eqn:finita} contradicts the optimality of $T_n$.



\begin{remark}\label{hmas}
In the spirit of the works \cite{White1999,depauwhardt,flat-relax}, we can replace  $x \mapsto |x|^\alpha$ with more general functions $H:\R\to[0,\infty)$ that are even, sub-additive, lower semi-continuous, monotone non-decreasing in $(0,+\infty)$, continuous in $0$ and satisfying $H(0)=0$. The associated functionals on traffic paths are usually called $H$-masses and are defined as
$$\Mass_H (T):=\int_E H(\theta(x)) d\Haus^1(x), \qquad \mbox{where $T=[E,\tau,\theta]\in \mathbf{R}_{1}(\R^d)$}.$$ 
The obvious analogue of Theorem \ref{thm:main} holds true. We divide the argument in two cases:
\begin{itemize}
\item {\em First case: $\lim_{\theta\to 0^+} \sfrac{H(\theta)}{\theta}=+\infty$.} For every $\delta>0$ there exists $\e(\delta,H)>0$ such that $\sfrac{\e(\delta,H)}{H(\e(\delta,H))}<\delta$.  One can repeat the proof of all the statements of Section \ref{sec:con} just changing $\MM$ with $\Mass_H$. The only differences are in Lemma \ref{second}: the statement \eqref{concl} becomes
$$
\mathbb{M}_H\left (T_{n}\trace \left(G_{k} \cap \left \{|\theta_{n}|>\e\left(\frac{\delta_{T,\Delta}}{2C},H\right)\right\}\right )\right) \geq  \mathbb{M}_H(T)-\Delta,
$$ 
in  the proof we choose $\e:=\e(\sfrac{\delta_{T,\Delta}}{2C},H)$ and we change \eqref{eqn:mass-to-0} in
$$\mathbb M(T_{n_i}\trace (G'_{k_i}\cap \{|\theta_{n_i}|\leq \e \}))<\frac{\e}{H(\e)}\mathbb M_H(T_{n_i}\trace (G'_{k_i} \cap \{|\theta_{n_i}|\leq \e\}))<C\frac{\e}{H(\e)}<\frac{\delta_{T,\Delta}}{2}.$$
We can then repeat verbatim Section \ref{pro}, with the same proof of Theorem \ref{thm:main}, just changing $\MM$ with $\Mass_H$ and modifying \eqref{conc2} according to the new version of Lemma \ref{second}.
\item {\em Second case: $\liminf_{\theta\to 0^+} \sfrac{H(\theta)}{\theta}<+\infty$.} Then it is easy to show that the minimal transport energy
$$W^H(\mu^-,\mu^+):=\inf\{\Mass_H(T): \mbox{$T$ is a traffic path connecting $\mu^-$ to $\mu^+$}\},$$
defined analogously to \eqref{mainp}, metrizes the weak-$*$ convergence of measures. We can then simply repeat the proof in \cite[Proposition 6.12]{BCM} to get the validity of Theorem \ref{thm:main}.
\end{itemize}
We observe moreover that the continuity of $H$ in $0$ is a necessary hypothesis for the validity of Theorem \ref{thm:main}. Indeed consider the case of the size, i.e.
\begin{equation}\label{size}
H(\theta)=1 \quad \mbox{on $\R\setminus \{0\}$ \qquad  and \qquad }H(0)=0.
\end{equation}
Consider $\mu^-:=\delta_0$ and $\mu^+:= \delta_{e_1}$; for every $n \in \N$ we define 
$$\mu^-_n:= \delta_0 \qquad \mbox{and}\qquad \mu^+_n:=\frac 1n \delta_{\sfrac{e_1}{2}+\sfrac{e_2}{8}}+\left(1-\frac 1n\right)\delta_{e_1}.$$
Since $\mu^-_n$ and $\mu^+_n$ are finite atomic measures, by \cite[Proposition 9.1]{BCM} the optimal traffic path $T_n$ is a finite graph made of segments with no loops. Moreover, by \eqref{size}, the energy is the sum of the length of the segments composing the graph.
In particular, the graph has to be connected, since both the points $\sfrac{e_1}{2}+\sfrac{e_2}{8}$ and $e_1$ have to be connected to $0$.
As a consequence, the energy of any traffic path in $\TP(\mu^-_n,\mu^+_n)$ must be bigger or equal than the length of the minimal tree connecting the three points, which is the union of the support of the following two curves $\gamma_1:[0,1]\to \R^d$
\begin{equation}
\gamma_1(t):= t \left(\frac{e_1}2+\frac  {e_2}8\right), \qquad \mbox{and} \qquad 
\gamma_2(t):= \frac {1+t}2 e_1+\frac {1-t}8e_2.
\end{equation}
Hence $W^H(\mu^-_n,\mu^+_n)=\frac{\sqrt{17}}{4}$ for every $n \in \N$ and an optimal traffic path $T_n\in \OTP(\mu^-_n,\mu^+_n)$ is
\begin{equation*}
T_n:=I_{\gamma_1}+(1-\sfrac{1}{n})I_{\gamma_2}.
\end{equation*}
We observe that
$$T_n\rightharpoonup T:= I_{\gamma_1}+I_{\gamma_2}.$$
As previously observed $\mathbb M_H(T)=\sfrac{\sqrt{17}}{4}>1\geq W^H(\mu^-,\mu^+)$ (since the segment joining $\mu^-$ and $\mu^+$ has energy one).
Since $\mu^\pm_n\rightharpoonup \mu^\pm$, this inequality contradicts the stability.

\end{remark}

%
%

 \subsection*{Acknowledgments}
	M. C. was partially supported by the Swiss National Science Foundation grant 200021\_182565. A. M. acknowledges partial support from GNAMPA-INdAM.

\nocite{}

%
%

\vskip .3 cm

{\parindent = 0 pt\begin{footnotesize}

Maria Colombo
\\
M.C. EPFL SB, Station 8, 
CH-1015 Lausanne, Switzerland
\\
e-mail M.C.: {\tt maria.colombo@epfl.ch}
\\
~
\\
Antonio De Rosa
\\
Courant Institute of Mathematical Sciences, New York University, New York, NY, USA
\\
e-mail A.D.R.: {\tt derosa@cims.nyu.edu}
\\
~
\\
Andrea Marchese
\\
Dipartimento di Matematica, Universit\`a degli Studi di Pavia, Pavia, Italy\\
e-mail A.M.: {\tt andrea.marchese@unipv.it}

\end{footnotesize}
}

\end{document}